\documentclass[12pt, a4paper]{amsart}
\usepackage{amssymb}
\usepackage{amsmath,esint,enumitem}
\usepackage{colortbl}
\usepackage[dvipsnames]{xcolor}
\usepackage[colorlinks, allcolors=Fuchsia,pdfstartview=,pdfpagemode=UseNone]{hyperref} %RedViolet
\usepackage{graphics,fullpage}
\usepackage{multirow}
\makeatletter
\@namedef{subjclassname@2020}{\textup{2020} Mathematics Subject Classification}
\makeatother

\newtheorem{theorem}[equation]{Theorem}%[section]
\newtheorem{lemma}[equation]{Lemma}
\newtheorem{proposition}[equation]{Proposition}

\theoremstyle{definition}
\newtheorem{definition}[equation]{Definition}
\newtheorem{assumption}[equation]{Assumption}

\theoremstyle{remark}

\numberwithin{equation}{section}

\let\oldmarginpar\marginpar
\renewcommand\marginpar[1]{\-\oldmarginpar[\raggedleft\footnotesize #1]%
{\raggedright\footnotesize #1}}

\newcommand{ \R }{ \mathbb{R} }

\renewcommand{\epsilon}{\varepsilon}
\renewcommand{\phi}{\varphi}
\renewcommand{\le}{\leqslant}
\renewcommand{\ge}{\geqslant}

\newcommand{\loc}{\mathrm{loc}}

\newcommand{\ainc}[1]{\hyperref[ainc]{{\normalfont(aInc){\ensuremath{_{#1}}}}}}
\newcommand{\adec}[1]{\hyperref[adec]{{\normalfont(aDec){\ensuremath{_{#1}}}}}}
\newcommand{\inc}[1]{\hyperref[inc]{{\normalfont(Inc){\ensuremath{_{#1}}}}}}
\newcommand{\dec}[1]{\hyperref[dec]{{\normalfont(Dec){\ensuremath{_{#1}}}}}}

\newcommand{\DMA}[1]{\hyperref[DMA]{{\normalfont(DMA1)\ensuremath{_{#1}}}}}

\newcommand{\aonen}[1]{\hyperref[aone]{{\normalfont(A1-\ensuremath{#1})}}}
\newcommand{\VAn}[1]{\hyperref[VA]{{\normalfont(VA1-\ensuremath{#1})}}}
\newcommand{\wVAn}[1]{\hyperref[wVA]{{\normalfont(wVA1-\ensuremath{#1})}}}
\newcommand{\VMAn}[1]{\hyperref[VMA]{{\normalfont(VMA1-\ensuremath{#1})}}}
\newcommand{\wVMAn}[1]{\hyperref[wVMA]{{\normalfont(wVMA1-\ensuremath{#1})}}}
\newcommand{\DMAn}[1]{\hyperref[DMA]{{\normalfont(DMA1-\ensuremath{#1})}}}
\newcommand{\wDMAn}[1]{\hyperref[wDMA]{{\normalfont(wDMA1-\ensuremath{#1})}}}
\newcommand{\SAn}[1]{\hyperref[SA]{{\normalfont(SA1-\ensuremath{#1})}}}

\begin{document}
\title[]
{Lipschitz regularity for orthotropic functionals with general growth}

\author{Mikyoung Lee}
  %  Address of record for the research reported here
  \address{Mikyoung Lee, Department of Mathematics and Institute of Mathematical Science, Pusan National University, Busan
46241, Republic of Korea}
%    Current address
\email{mikyounglee@pusan.ac.kr}

\author{Jihoon Ok}
 % Address of record for the research reported here
\address{Jihoon Ok, Department of Mathematics, Institute for Mathematical and Data Science, Sogang University, Seoul 04107, Republic of Korea}
% Current address
\email{\texttt{jihoonok@sogang.ac.kr}}

\author{Bianca Stroffolini}
 % Address of record for the research reported here
\address{Bianca Stroffolini, Department of Mathematics and Applications “R. Caccioppoli”, University of
Naples Federico II, Via Cintia, Monte S. Angelo, 80126 Naples, Italy}
% Current address
\email{\texttt{bstroffo@unina.it}}
% \thanks will become a 1st page footnote.
%\thanks{}

% General info
\date{\today}


\begin{abstract}
We study the local Lipschitz regularity of local minimizers for a class of degenerate orthotropic functionals with
%general Orlicz growth.
 $\phi$-growth, where $\phi$ is a general N-function.
Unlike standard 
%functionals with radial growth, 
isotropic functionals,
the ellipticity of the associated Euler-Lagrange equation degenerates separately in each coordinate direction, presenting 
%severe anisotropic challenges. 
significant anisotropic difficulties.
Furthermore, the general N-function setting lacks the algebraic scale invariance available in the classical orthotropic $p$-Laplacian case.
Despite these structural difficulties, we prove that local minimizers are locally Lipschitz continuous. Our approach relies on a regularized approximation scheme, mixed-direction Caccioppoli inequalities, and a carefully designed Moser-type iteration that incorporates an interpolation argument to bridge the gaps between consecutive integrability exponents.
\end{abstract}

\maketitle

%%%%%%%%%%%%%%%%%%%%%%%%%%%%%%%%%%%%%%%%%%%%%%%%%%%%%%%%%%%%%%%%%%%%%%%%%%%%%%%%%%

\section{Introduction}\label{sect:intro}
%%%%%%%%%%%%%%%%%%%%%%%%
%%%%%%%%%%%%%%%%%%%%%%%%
%%%%%%%%%%%%%%%%%%%%%%%%

In this paper, we study the local regularity of local minimizers $u \in W^{1,\phi}_{\loc}(\Omega)$ of the functional
\begin{equation}\label{main_functional}
\mathfrak{F}_\phi(u ; \Omega):= \sum_{i=1}^n \int_{\Omega} \phi(|u_{x_i}|)\,dx,
\end{equation}
where $\Omega$ is an open set in $\R^n$ with $n \ge 2$, and $\phi:[0,\infty)\to[0,\infty)$ is an N-function %satisfying suitable two-sided growth assumptions. a  $C^1$-function
 satisfying suitable growth conditions, see Assumption~\ref{assump:growth2}.
 We first recall the variational notion of solution considered throughout the paper.
 \begin{definition}
 A function $u \in W^{1,\phi}_{\loc}(\Omega)$ is called a local minimizer of $\mathfrak{F}_\phi$ if it minimizes the functional \eqref{main_functional} under compactly supported perturbations. Namely, for every ball $B \subset \Omega$, the following inequality holds:
% \begin{equation}\label{localmini}
\[
 \mathfrak{F}_\phi(u ; B)\le \mathfrak{F}_\phi(u+\zeta ; B), \quad  \forall \zeta \in C^1_0(B).
 \]
% \end{equation}
 \end{definition}
%%%%%%%%%%%%

Under the growth conditions imposed on $\phi$, we note that the local minimizer $u$ of the functional \eqref{main_functional} is also a local weak solution of the associated orthotropic $\phi$-Laplace equation 
\begin{equation}\label{mainPDE}
\sum_{i=1}^n \left(\frac{\phi' (|u_{x_i}|)}{|u_{x_i}|}u_{x_i}\right)_{x_i}= 0 \ \ \text{ in } \Omega. 
\end{equation}
Equation \eqref{mainPDE} may be regarded as the orthotropic counterpart of the isotropic $\phi$-Laplace equation
\[
\operatorname{div}
\left(
\frac{\phi'(|\nabla u|)}{|\nabla u|}\nabla u
\right)
=0.\]
Although these two equations share comparable growth properties, their ellipticity structures are substantially different. This distinction can be seen as follows. 

Assume that $\phi\in C^{2}((0,\infty))$ satisfies the growth condition \eqref{cond:growth2-2}. For the isotropic integrand
\[
H(\xi):=\phi(|\xi|), \qquad \xi\in\mathbb{R}^{n},
\]
we have that, for $\xi\neq0$ and every $\eta\in\mathbb{R}^{n}$,
\[
\left\langle D^{2}H(\xi)\eta,\eta\right\rangle = \frac{\phi'(|\xi|)}{|\xi|} \left( |\eta|^{2} - \frac{(\xi\cdot\eta)^{2}}{|\xi|^{2}} \right)+ \phi''(|\xi|) \frac{(\xi\cdot\eta)^{2}}{|\xi|^{2}}.
\]
Thus, the radial eigenvalue (that is, when $\xi\cdot\eta=|\xi| |\eta|$)  is $\phi''(|\xi|)$, while the tangential eigenvalues (that is, when $\xi\cdot\eta=0$) are equal to $\phi'(|\xi|)/|\xi|$. Due to \eqref{cond:growth2-2},
these eigenvalues are comparable, and the ellipticity in every direction is governed by a single function depending only on $|\xi|$.
By contrast, 
for the orthotropic integrand 
\[
G(\xi) := \sum_{i=1}^{n}\phi(|\xi_i|), \qquad \xi=(\xi_1,\ldots,\xi_n)\in\mathbb{R}^{n},
\]
we have that for $\xi=(\xi_1,\ldots,\xi_n)\in\R^n$ with $\xi_i\neq0$ for all $i$ and every $\eta\in \R^n$,
\[
\left\langle D^{2}G(\xi)\eta,\eta\right\rangle = \sum_{i=1}^{n} \phi''(|\xi_i|)\eta_i^{2}.
\]
Thus, unlike  the isotropic case, the ellipticity in the $i$-th coordinate direction is determined by the individual component $|\xi_i|$, rather than by $|\xi|$. In particular, one direction may have very weak ellipticity even when $|\xi|$ is large.

The regularity theory for elliptic equations and variational integrals with general isotropic growth is by now well developed. Lieberman extended the classical regularity theory of Ladyzhenskaya and Uraltseva to equations governed by a general growth function $\phi$~\cite{Lie91}. Diening, the third author of the paper and Verde obtained regularity results for $\phi$-Laplace systems~\cite{DieStrVer09}. Very recently, Antonini established interior and global regularity theory for quasilinear uniformly elliptic equations with Orlicz growth \cite{Antonini26}. Note that the previous papers mainly consider the H\"older continuity of the gradient of weak solutions.
Regarding the Lipschitz regularity, it has been known that one can consider weaker ellipticity and growth conditions. A major advance was made by Marcellini and Papi~\cite{MarPap06}, who established local Lipschitz regularity for local minimizers of variational functionals under growth assumptions sufficiently general to include regimes ranging from linear to exponential growth. Esposito, Mingione, and Trombetti later proved local Lipschitz regularity for broad classes of scalar elliptic equations and variational integrals with N-function growth, under weak differentiability and monotonicity assumptions together with suitable generalized uniform ellipticity conditions~\cite{EspMinTro06}.

While previous results provide an important framework for variational problems with Orlicz growth, their ellipticity essentially relies on the modulus of the gradient. However, this property is absent in the orthotropic functional \eqref{main_functional}, since its ellipticity degenerates separately in each coordinate direction. This structural difference highlights a significant gap between  regularity theories for the isotropic and orthotropic functionals. In this paper, we fill this gap by establishing the local Lipschitz regularity. 

\begin{theorem}\label{thm:main}
Let $u\in W^{1,\phi}_{\loc}(\Omega)$ be a local minimizer of the functional $\mathfrak{F}_\phi$. If $\phi \in C^1([0,\infty))$ satisfies Assumption~\ref{assump:growth2} with $2\le p \le q$, then $u$ is locally Lipschitz in $\Omega$. Moreover, for any ball $B_{2r}\Subset \Omega$, we have the estimate
\begin{equation}\label{estimate:Lip}
\| \nabla u\|_{L^{\infty}(B_{r/2})}\le  c\, \phi^{-1} \bigg( \fint_{B_{2r}} \phi(|\nabla u|)\,dx \bigg),
\end{equation}
where $c=c(n,p,q)>0$ is a constant independent of $u$ and $r$.
\end{theorem}

For the model case $\phi(t)=t^{p}/p$, estimate \eqref{estimate:Lip} reduces to the natural local gradient bound for orthotropic $p$-harmonic functions. 
The study of the local regularity for minimizers of orthotropic functionals goes back to the pioneering work of Ural'tseva and Urdaletova \cite{UU}, who used Bernstein-type techniques to establish gradient bounds for the classical functional 
%\begin{equation}\label{power0}
\[
\mathfrak{F}_0(u ; \Omega):= \sum_{i=1}^n \int_{\Omega} |u_{x_i}|^p \,dx.
\]
%\end{equation}
 In recent years, a widely degenerate variant of the orthotropic functional \cite{bracarliersant10,bracarlier14}:
%\begin{equation}\label{powerdelta}
\[
\mathfrak{F}_{\delta}(u ; \Omega):= \sum_{i=1}^n \int_{\Omega} (u_{x_i}-\delta_i)_+^p \,dx
\]
%\end{equation}
has emerged in connection with variational models for traffic and optimal transport flows. Earlier works established only almost Lipschitz regularity for minimizers, meaning that the gradients of the solutions were shown to be integrable to any power $\gamma<\infty$. A major advance was achieved by Bousquet, Brasco, and  Julin  \cite{BouBraJul16}, who proved genuine Lipschitz regularity. 
 See also \cite{BouBra18} for a $C^1$-regularity result in the plane.
 However, while their result held for any $p \ge 2$ in the two-dimensional setting, it was restricted to $p \ge 4$ in higher dimensions.

A complete generalization of the previous papers occurred with the paper \cite{BouBraLeoVer18}, where the Lipschitz regularity was proved in all dimensions and with a forcing term.
The proof is not merely an adaptation of standard $p$-Laplacian techniques. The key novelty is a new family of mixed-direction Caccioppoli inequalities and a carefully designed Moser iteration scheme that couples different gradient components. These tools allow the authors to overcome the severe orthotropic degeneracy that had restricted earlier results to lower dimensions.
We also refer to Lipschitz regularity results  \cite{BouBra20,BouBraLeo24,Lieberman05} for stronger orthotropic functionals $\int_\Omega \sum_{i=1}^n |u_{x_i}|^{p_i}\,dx$ with $1< p_i<\infty$, \cite{CupMarMas09} for orthotropic functionals with variable exponents $\int_\Omega \sum_{i=1}^n |u_{x_i}|^{p_i(x)}\,dx$, and \cite{ALM2026}  for double phase orthotropic functionals  $\int_\Omega \sum_{i=1}^n |u_{x_i}|^{p}+a(x)|u_{x_i}|^{q}\,dx$ with $2\le p\le q \le p(1+\frac{\alpha}{n})$ and $a\in W^{1,\frac{n}{1-\alpha}}$. 
%%%%%
The purpose of this paper is to extend the Lipschitz regularity theory to the general Orlicz setting. 
We assume that the growth function is the same in all coordinate directions, but we replace the power function $t^{p}$ by a general N-function $\phi \in C^1([0,\infty))$ satisfying a $(p,q)$-growth condition. 
This setup includes not only the standard power case but also nonhomogeneous examples like $\phi(t) = \frac{1}{2}(t^p + t^q)$ or $\max\{t^p,t^q\}$ for $2 \le p < q$, and functions whose growth rates vary between $p$ and $q$.

Moving from power growth to general Orlicz growth brings a structural difficulty: the loss of homogeneity. 
In the power case, exact homogeneity gives nice algebraic identities, which allow us to control the energy estimates easily using fixed powers of the derivatives. 
For a general N-function $\phi$, we no longer have such identities. 
Therefore, the nonlinear terms arising from the differentiated equation must be controlled by using the lower and upper growth indices of $\phi$, instead of a single power exponent. 
Moreover, since the N-function $\phi$ is not smooth, or even $C^2$, we need to apply a very delicate approximation argument.

We follow the strategy developed in~\cite{BouBraLeoVer18}, while facing the additional difficulty of dealing with functions that lack homogeneity at any scale. 
Specifically, we approximate the local minimizer by smooth solutions $u_{\epsilon}$ of uniformly elliptic problems, ensuring that the resulting estimates on $u_{\epsilon}$ are stable with respect to $\epsilon$. 
By differentiating the regularized equation with respect to $x_j$, we establish a new family of Caccioppoli inequalities. 
We then initiate a Moser-type iteration for convex powers of the partial derivatives $u_{x_k}$. 
Finally, we apply Sobolev inequality and we fill the holes of the Moser's scheme.

The paper is organized as follows. In Section~\ref{sect:preliminaries}, we collect the basic properties of N-functions, Orlicz-Sobolev spaces, and the higher integrability estimates needed throughout the proof. We then introduce the uniformly elliptic regularized problems and establish energy estimates that are uniform in the regularization parameter. The subsequent sections are devoted to the differentiated equations, the standard and mixed-direction Caccioppoli inequalities, and the finite staircase argument. Finally, we perform the Moser iteration with hole-filling and pass to the limit in the approximation scheme to conclude the proof.

\section{Preliminaries and notation}\label{sect:preliminaries}
Let $f,g:E\to\mathbb{R}$ be measurable functions, where $E\subset\mathbb{R}^n$. 
For any measurable set $E$ satisfying $0<|E|<\infty$, we write
\[
(f)_E := \fint_E f\,dx = \frac{1}{|E|}\int_E f\,dx
\]
to denote the average of $f$ over $E$.
We write $f\lesssim g$ if there exists a constant $C>0$ such that
\[
f(y)\le Cg(y) \qquad \text{for all } y\in E.
\]
Moreover, $f\approx g$ means that both $f\lesssim g$ and $g\lesssim f$ hold.

Let $E\subset\mathbb{R}$. A function $f:E\to\mathbb{R}$ is said to be 
\textit{almost increasing} on $E$ with constant $L\ge1$ if
\[
f(s)\le Lf(t)
\qquad \text{whenever } s,t\in E \text{ and } s\le t.
\]
If the above inequality holds with $L=1$, then $f$ is said to be 
\textit{increasing} on $E$.
The notions of \textit{almost decreasing} and \textit{decreasing} are defined analogously.

\subsection{Orlicz functions} Throughout this paper, let $\phi:[0,\infty)\to[0,\infty)$ be an N-function satisfying
$
\phi(0)=0.
$
We assume that $\phi$ admits a right-continuous derivative $\phi'$, that $\phi'$ is increasing, and that
$\phi'(0)=0,\ 
\phi'(t)>0$ for all $t>0.$
For convenience, we normalize $\phi$ by assuming that $\phi(1)=1$. 
Otherwise, the constants appearing in the estimates may additionally depend on $\phi(1)$.

Moreover, we impose the following growth condition on $\phi$.

\begin{assumption}\label{assump:growth}
There exist constants $1< p\le q<\infty$ such that $\frac{\phi(t)}{t^p}$ is almost increasing and $\frac{\phi(t)}{t^q}$ is almost decreasing for $t \in (0,\infty)$ with constant $L\ge 1$. That is, the following holds for any $t >0$ and $\lambda\ge 1$:
\begin{equation}\label{cond:growth-1}
L^{-1}\lambda^p \phi(t) \le \phi(\lambda t ) \le L \lambda^q \phi(t).
\end{equation}
\end{assumption}

The conjugate function associated with $\phi$ is defined by
\[
\phi^*(t):=\sup_{s\ge 0 } \big( st-\phi(s)\big).
\]
By definition, Young's inequality states that for every $s,t\ge0$,
\[
st\le \phi(t)+\phi^*(s).
\]
Assumption~\ref{assump:growth} implies that both $\phi$ and $\phi^*$ satisfy the $\Delta_2$-condition. Consequently,
\begin{equation}\label{equivalent}
\phi^*\bigg( \frac{\phi(t)}{t}\bigg) \sim \phi^*(\phi'(t)) \sim \phi(t), 
\quad \text{and thus}\quad
t \phi'(t) \sim \phi(t),
\end{equation}
where the implicit constants depend only on $p$, $q$, and $L$; see \cite[Theorem~2.4.10]{HarH19}.

For the regularity theory developed later, we will consider the following stronger assumption:
%We will consider the following stronger assumption:
\begin{assumption}\label{assump:growth2}
Let $\phi\in C^{1}([0,\infty))$. There exist constants $1<p\le q <\infty$ such that $\frac{\phi'(t)}{t^{p-1}}$ is increasing and $\frac{\phi'(t)}{t^{q-1}}$ is  decreasing for $t \in (0,\infty)$. That is, the following holds for any $t >0$ and $\lambda\ge 1$:
%\begin{equation}\label{cond:growth2-1}
\[
\lambda^{p-1} \phi'(t) \le \phi'(\lambda t ) \le  \lambda^{q-1} \phi'(t).
\]
%\end{equation}
\end{assumption}
Indeed, the previous assumption will be preserved under convolution and future approximations.

We note that
if $\phi\in C^1([0,\infty))\cap C^{2}((0,\infty))$,  then Assumption~\ref{assump:growth2} is equivalent to the following inequality:
\begin{equation}\label{cond:growth2-2}
0<p-1 \le \frac{t\phi''(t)}{\phi'(t)} \le q-1 \quad \text{for all } \ t>0.
\end{equation}
Moreover, Assumption~\ref{assump:growth2} implies Assumption~\ref{assump:growth} with $L=1$ and the same exponents $p$ and $q$.

In the classical theory of Orlicz spaces, Assumption~\ref{assump:growth} is known as $\Delta_2$-condition
for both $\phi$ and its conjugate function $\phi^*$. 
 In this setting, 
 the Orlicz space $L^\phi(\Omega)$ is defined as the set of all measurable functions $f:\Omega \rightarrow \R$ satisfying
\[
\int_{\Omega} \phi(|f(x)|)\,dx < \infty.
\]
Under the above assumption, $L^\phi(\Omega)$ is a separable and reflexive Banach space.
The associated Orlicz-Sobolev space $W^{1,\phi}(\Omega)$  is defined as the set of all functions $f \in L^{\phi}(\Omega) \cap W^{1,1}(\Omega)$ satisfying 
\[
\int_{\Omega} \phi(|\nabla f(x)|)\,dx < \infty.
\]

\subsection{Higher integrability}
Suppose that the N-function $\phi$ satisfies Assumption~\ref{assump:growth}. We define the nonlinear operator $A=(A_1,A_2,\cdots,A_n):\R^n\to\R^n$ by
$$
A_i(\xi):= \frac{\phi'(|\xi_i|)}{|\xi_i|}\xi_i,
\qquad i=1,2,\cdots,n,
$$
where $\xi=(\xi_1, \xi_2, \cdots, \xi_n) \in \R^n$. Then
by \eqref{cond:growth-1} and \eqref{equivalent}, the operator $A$ satisfies the standard $\phi$-growth and coercivity conditions:
$$
|A(\xi)| \le \bigg(\sum_{i=1}^n \phi'(|\xi_i|)^2\bigg)^{\frac{1}{2}}\le \tilde L \phi'(|\xi|)
$$
and
$$
A(\xi)\cdot\xi = \sum_{i=1}^n \phi'(|\xi_i|)|\xi_i|  \ge \tilde\nu \phi(|\xi|) 
$$
for some constants $0<\tilde \nu \le \tilde L$ depending on $n$, $p$, $q$, and $L$ .  

Note that the orthotropic equation \eqref{mainPDE} can be  rewritten in divergence form as
$$
\sum_{i=1}^n \left(\frac{\phi' (|u_{x_i}|)}{|u_{x_i}|}u_{x_i}\right)_{x_i}= \mathrm{div} A(Du)=0 
\ \ \text{ in } \Omega.
$$
Therefore, by standard regularity theory for the above problem (see, e.g.,  \cite{CiaFus99}), we have the following higher integrability estimates.

\begin{lemma}\label{lem:high}
Let $\phi$ be an N-function satisfying Assumption~\ref{assump:growth} and $u \in W^{1,\phi}_{\loc}(\Omega)$ be  a local weak solution of \eqref{mainPDE}. Then, there exists $\sigma=\sigma(p,q,n)>0$ such that $\phi(|\nabla u|)\in L^{1+\sigma}_{\loc}(\Omega)$  with the estimate: for any $B_{2r}\Subset \Omega$,
%\begin{equation}\label{eq:high}
\[
\fint_{B_r} \phi(|\nabla u|)^{1+\sigma}\,dx \le c \left(\fint_{B_{2r}} \phi(| \nabla u|)\,dx\right)^{1+\sigma}
\]
%\end{equation}
for some $c=c(n,p,q)>0$.
\end{lemma}

\section{Regularization}
Suppose that $\phi\in C^1([0,\infty))$ satisfies Assumption~\ref{assump:growth2}. We regularize the function $\phi$  in two steps. First, we use mollification to obtain a  smooth function $\tilde \phi_\epsilon$ for $\epsilon>0$; secondly, we deal with a modified shifted function in order to obtain a non-degenerate function $\psi_\epsilon$. It turns out that these regularizations are such that the corresponding equations satisfy the same exponents in Assumption~\ref{assump:growth2}.

Let $\epsilon\in(0,1)$ be small. We first define the regularized function $\tilde\phi_\epsilon$, associated with $\phi$, by 
\[
\tilde\phi'_{\epsilon}(t) 
:=  \int_0^\infty \phi'(s t) \varrho_{\epsilon}(s-1)\,ds 
= \int_0^\infty \phi'(s)\varrho_{\epsilon t}(s-t)\, ds
\quad\text{and}\quad
\tilde\phi_\epsilon(t):= \int_0^t \tilde\phi'_{\epsilon}(s)\,ds,
\]
where $\varrho\in C^\infty_0(\R)$ satisfies $\varrho\ge 0$, $\mathrm{supp}\, \varrho \subset (0,1)$ and $\|\varrho\|_{L^1(\R)} =1$, and $\varrho_{r}(t):=\frac{1}{r}\varrho(\frac{t}{r})$. Then we can infer  that $\tilde\phi_\epsilon\in C^1([0,\infty))\cap C^\infty((0,\infty))$.

Using the same argument as in \cite[Proposition~5.10]{HasO22}, we can obtain the following lemma. 

\begin{lemma} \label{lem:tphi}
Let $\tilde \phi_\epsilon \in  C^1([0,\infty))\cap C^\infty((0,\infty))$ be defined as above.
\begin{itemize}
\item[$\mathrm{(i)}$] $\tilde\phi_\epsilon$ satisfies Assumption~\ref{assump:growth2}, and thus \eqref{cond:growth2-2}, with $\phi$ replaced by $\tilde\phi_\epsilon$.
\item [$\mathrm{(ii)}$] $\phi'(t)\le \tilde\phi'_{\epsilon}(t) \le  (1+\epsilon)^{q-1}\phi'(t)$, and hence $\phi(t)\le \tilde\phi_{\epsilon}(t) \le  (1+\epsilon)^{q-1}\phi(t)$,  for all $t\ge0$.
\end{itemize}\end{lemma}
\begin{proof}
(i) It suffices to show that   $\lambda^{p-1}\tilde\phi'_{\epsilon}(t)\le \tilde\phi'_{\epsilon}(\lambda t)\le \lambda^{q-1}\tilde\phi'_{\epsilon}(t)$ for all $\lambda\ge 1$.  Since $\lambda^{p-1}\phi'(t)\le \phi'(\lambda t)\le \lambda^{q-1}\phi'(t)$ for $\lambda\ge1$,
 we have that \[
\tilde\phi'_{\epsilon}(\lambda t) = \int_0^\infty \phi'(s \lambda t) \varrho_{\epsilon}(s-1)\,ds 
\ge \lambda^{p-1}\int_0^\infty \phi'(s t) \varrho_{\epsilon}(s-1)\,ds = \lambda^{p-1}  \tilde\phi'_{\epsilon}(t)
\]
and
\[
\tilde\phi'_{\epsilon}(\lambda t) = \int_0^\infty \phi'(s \lambda t) \varrho_{\epsilon}(s-1)\,ds 
\le \lambda^{q-1}\int_0^\infty \phi'(s t) \varrho_{\epsilon}(s-1)\,ds = \lambda^{q-1} \tilde \phi'_{\epsilon}(t).
\]

(ii) Since $\mathrm{supp}\,\varrho_\epsilon \subset(0,\epsilon)$ with $\int_0^\epsilon \varrho_\epsilon(s)\,ds=1$,  $\phi'(t)$ is increasing in $t$, and $\tilde\phi'_{\epsilon}(t)/t^{q-1}$ is decreasing in $t$  by (i), we have that
\[
\tilde \phi'_{\epsilon}( t) = \int_{1}^{1+\epsilon} \phi'(s t) \varrho_{\epsilon}(s-1)\,ds \ge \phi'(t) \int_{1}^{1+\epsilon} \varrho_{\epsilon}(s-1)\,ds =\phi'(t),
\]
and 
\[
\tilde\phi'_{\epsilon}( t)  \le \phi'((1+\epsilon)t) \int_{1}^{1+\epsilon} \varrho_{\epsilon}(s-1)\,ds =\phi'((1+\epsilon)t)\le (1+\epsilon)^{q-1}\phi'(t).%\qedhere
\]
The resulting estimates for $\phi$ and $\tilde\phi_\epsilon$ are obtained directly by integrating the above inequalities.

\end{proof}

For $t\in\mathbb{R}$, we further define
\begin{equation}\label{def:g}
\psi_{\epsilon}'(t) := \frac{\tilde\phi'_{\epsilon}\big((\epsilon^2 + t^2)^{\frac12}\big)}{(\epsilon^2 + t^2)^{\frac12}} t 
\quad\text{and}\quad
\psi_{\epsilon}(t) := \int_{0}^{t} \psi_{\epsilon}'(s)\,ds= \tilde\phi_\epsilon\big((\epsilon^2 + t^2)^{\frac12}\big) - \tilde\phi_{\epsilon}(\epsilon).
\end{equation}
Note that $\psi_\epsilon$ is a smoothed version of the shifted N-function introduced in \cite{DieEtt08}, corresponding to $\tilde\phi_\epsilon$ with shift $\epsilon$. 
By the definition of $\psi_\epsilon$ and Lemma~\ref{lem:tphi}(ii), we infer that $v\in W^{1,\psi_\epsilon}(U)$ if and only if $v\in W^{1,\phi}(U)$ for any bounded open set $U\subset \R^n$.
Moreover, we directly see that 
\begin{equation}\label{eq:g'}
|\psi_{\epsilon}'(t)| = \frac{\tilde\phi'_{\epsilon}\big((\epsilon^2 + t^2)^{\frac12}\big)}{(\epsilon^2 + t^2)^{\frac12}} |t| \le  \tilde\phi'_{\epsilon}\big((\epsilon^2 + t^2)^{\frac12}\big)\le  \tilde\phi'_{\epsilon}\big(\epsilon + |t| \big)
\end{equation}
and that, since
\[\psi_{\epsilon}''(t) =\left\{\left( \frac{\tilde\phi''_{\epsilon}\big((\epsilon^2 + t^2)^{\frac12}\big)(\epsilon^2 + t^2)^{\frac12}}{\tilde\phi'_{\epsilon}\big((\epsilon^2 + t^2)^{\frac12}\big)} -1 \right)\frac{t^2}{\epsilon^2 + t^2}+ 1\right\} \frac{\tilde\phi'_{\epsilon}\big((\epsilon^2 + t^2)^{\frac12}\big)}{(\epsilon^2 + t^2)^{\frac12}} ,
\]
by using \eqref{cond:growth2-2} with $\phi$ replaced by $\tilde\phi_\epsilon$,  
\begin{equation}\label{eq:g''}
\min\{p-1,1\}\frac{\tilde\phi'_{\epsilon}((\epsilon^2 + t^2)^{\frac12})}{(\epsilon^2 + t^2)^{\frac12}} \le \psi_{\epsilon}''(t) \le \max\{q-1,1\}\frac{\tilde\phi'_{\epsilon}((\epsilon^2 + t^2)^{\frac12})}{(\epsilon^2 + t^2)^{\frac12}},
\end{equation}
and hence
\begin{equation}\label{eq:g'g''}
\min\{p-1,1\} \le \frac{t \psi_{\epsilon}''(t) }{\psi_{\epsilon}'(t)} \le \max\{q-1,1\},
\end{equation}
for any  $t\in \R\setminus\{0\}$.

We are now approximating the original problem with a sequence of  minimizers related to $\epsilon$-regularized functionals. It turns out that the minimizers $u_{\epsilon}$ have the gradients locally convergent in $L^{\phi}$ to the minimizer $u$ of the original problem. \par
Let $u\in W^{1,\phi}_{\loc}(\Omega)$ be a local minimizer of \eqref{main_functional}, $B_{2r} \Subset \Omega$,  and $\epsilon_0:= \min\{1,\mathrm{dist}(B_{2r},\partial \Omega)\}$. For  $\epsilon\in(0,\epsilon_0)$,
%satisfying \eqref{assump:normal}. 
we consider the
minimizer $u^\epsilon\in U^{\epsilon}+W^{1,\phi}_0(B_{2r})$ of the functional 
\begin{equation}\label{eq:regularizedenergy}
\mathfrak{F}_{\psi_\epsilon}(v ; B_{2r}):= \sum_{i=1}^n \int_{B_{2r}}\psi_\epsilon ( v_{x_i}) \, dx, 
\quad  v\in U^{\epsilon}+W^{1,\phi}_0(B_{2r}),
\end{equation}
where $U^\epsilon(x):=u*m_\epsilon (x)$ for $x\in B_{2r}$ and $m_\epsilon(x):=\epsilon^{-n}m(x/\epsilon)$ with $\epsilon>0$ is the standard mollifier supported in $B_\epsilon(0)$. 
Note that  $U^\epsilon \to u$ as $\epsilon\to 0^+$ in $L^\phi(B_{2r})$, and $u^\epsilon$ is the unique weak solution to the following regularized and nondegenerate problem:
\begin{equation}\label{PDE_regularized_BC}
\sum_{i=1}^n \big(\psi_{\epsilon}'(u^\epsilon_{x_i})\big)_{x_i}= 0 \ \ \text{ in } \ B_{2r}, 
\quad\text{and}\quad
u^\epsilon=U^\epsilon=u*m_\epsilon \ \ \text{ on } \ \partial B_{2r}.
\end{equation}

We observe that \begin{equation}\label{eq:energy_estimate0}
\sum_{i=1}^n \int_{B_{2r}}\psi_\epsilon ( u^{\epsilon}_{x_i}) \, dx \le  \sum_{i=1}^n \int_{B_{2r}}\psi_\epsilon ( U^{\epsilon}_{x_i}) \, dx  \le   (1+\epsilon)^{q-1} \sum_{i=1}^n \int_{B_{2r}}\phi (\epsilon+ |U^\epsilon_{x_i}|) \, dx,
\end{equation}
where the first inequality follows from the minimality of $u^\epsilon$,
while the second one follows from the definition of $\psi_\epsilon$
in \eqref{def:g} and Lemma~\ref{lem:tphi}(ii).
%by the minimality of $u^\epsilon$, the definition of $\psi_\epsilon$ in \eqref{def:g} and Lemma~\ref{lem:tphi}(ii).
Hence, by Lemma~\ref{lem:tphi}(ii),  \eqref{def:g}, \eqref{eq:energy_estimate0}, and the strong convergence of $U^\epsilon$, we infer that 
\begin{equation}\label{eq:energy_estimate1}
\begin{aligned}
\int_{B_{2r}}\phi (|\nabla u^\epsilon|)\,dx 
& \le  \int_{B_{2r}}\tilde\phi_\epsilon (|\nabla u^\epsilon|)\,dx  
 \le \int_{B_{2r}}\psi_\epsilon (|\nabla u^\epsilon|)+\tilde\phi_\epsilon(\epsilon) \, dx \\ 
& \lesssim  \int_{B_{2r}}\phi (|\nabla U^\epsilon|)+1\,dx \lesssim \int_{B_{2r}}\phi (|\nabla u|)+1\,dx
\end{aligned}
\end{equation}
for any sufficiently small $\epsilon\in(0,1)$.
Then we have the following convergence result.
\begin{proposition}[Convergence to a weak solution] \label{prop:strongconv}
Under the above setting, we have that $\nabla u^\epsilon \to \nabla u$ in $L^\phi(B_{2r},\R^n)$ as $\epsilon\to 0^+$.
\end{proposition}
\begin{proof}  
Set $\int_{B_{2r}}\phi(|\nabla u |)+1\,dx=:M$. By taking $u^\epsilon-U^\epsilon$ as the test function in the weak formulations of \eqref{mainPDE} and \eqref{PDE_regularized_BC}, we have 
\[
\sum_{i=1}^n \int_{B_{2r}} \left(\psi_{\epsilon}' (u^\epsilon_{x_i})- \frac{\phi'(|u_{x_i}|)}{|u_{x_i}|}u_{x_i} \right)(u^{\epsilon}_{x_i}-U^\epsilon_{x_i}) \, dx =0,
\]
and hence
\[\begin{split}
& \sum_{i=1}^n \int_{B_{2r}} \left(\psi_{\epsilon}' (u^\epsilon_{x_i}) -\psi_{\epsilon}' (u_{x_i})\right)(u^{\epsilon}_{x_i}-u_{x_i}) \, dx\\
& = \sum_{i=1}^n \int_{B_{2r}}\left(\psi_{\epsilon}' (u^\epsilon_{x_i})-\frac{\phi'(|u_{x_i}|)}{|u_{x_i}|}u_{x_i} \right)(U^\epsilon_{x_i}-u_{x_i}) \, dx \\
&\qquad 
+\sum_{i=1}^n \int_{B_{2r}}\left(\frac{\phi'(|u_{x_i}|)}{|u_{x_i}|}u_{x_i}  -\psi_{\epsilon}' (u_{x_i})\right)(u^\epsilon_{x_i}-u_{x_i}) \, dx \\
&=: I_1+I_2.
\end{split}\]

We first estimate $I_1$ as $\epsilon \to 0^+$. 
By H\"older's inequality,
\[
I_1\le c \sum_{i=1}^n \left\| \psi_{\epsilon}' (u^\epsilon_{x_i})-\frac{\phi'(|u_{x_i}|)}{|u_{x_i}|}u_{x_i}\right\|_{L^{\phi^*}(B_{2r})} \|U^\epsilon_{x_i}-u_{x_i}\|_{L^\phi(B_{2r})}.
\]
Note that by \eqref{eq:g'}, Lemma~\ref{lem:tphi}(ii), and \eqref{equivalent},
\[\begin{split}
\phi^*\left(\left|\psi_{\epsilon}' (u^\epsilon_{x_i})-\frac{\phi'(|u_{x_i}|)}{|u_{x_i}|}u_{x_i}\right|\right)  &\le  c \phi^*\left(\tilde\phi'_{\epsilon} \big(\epsilon +|u^{\epsilon}_{x_i}|\big) +\phi' \big(|u_{x_i}|\big)\right)\\
&\le c \phi(|u^{\epsilon}_{x_i}|) + c \phi(|u_{x_i}|) + c\phi(\epsilon),
\end{split}\]
and hence, by \eqref{eq:energy_estimate1}, 
\[
\int_{B_{2r}} \phi^*\left(\left|\psi_{\epsilon}' (u^\epsilon_{x_i})-\frac{\phi'(|u_{x_i}|)}{|u_{x_i}|}u_{x_i}\right|\right)\,dx \le c \int_{B_{2r}}\phi(|\nabla u |)+1\,dx \le cM,
\]
which implies 
\[
\left\| \psi_{\epsilon}' (u^\epsilon_{x_i})-\frac{\phi'(|u_{x_i}|)}{|u_{x_i}|}u_{x_i}\right\|_{L^{\phi^*}(B_{2r})}\le c M
\]
for all sufficiently small $\epsilon \in (0,\epsilon_0)$  and $  i=1,2,\dots,n$. Therefore, since $U^\epsilon \to u $ as $\epsilon \to 0^+$ in $L^\phi(B_{2r})$, we have that $I_1\to 0$ as $\epsilon \to 0^+$.

We next consider $I_2$. Observe that for every $\epsilon\in (0,1)$
\[\begin{split}
&\phi^*\left(\left|\frac{\phi'(|u_{x_i}|)}{|u_{x_i}|}u_{x_i} -\psi_{\epsilon}' (u_{x_i})\right|\right) \\
&\le c \phi^*\left(\left|\frac{\phi'(|u_{x_i}|)}{|u_{x_i}|}u_{x_i}-\frac{\tilde\phi'_{\epsilon}(|u_{x_i}|)}{|u_{x_i}|}u_{x_i}\right|\right) +c \phi^*\left(\left|\frac{\tilde\phi'_{\epsilon}(|u_{x_i}|)}{|u_{x_i}|}u_{x_i} -\frac{\tilde\phi'_{\epsilon}\big((\epsilon^2+|u_{x_i}|^2)^{\frac{1}{2}}\big)}{(\epsilon^2+|u_{x_i}|^2)^{\frac{1}{2}}}u_{x_i}\right|\right).
\end{split}\]
Then, by the continuity of the functions $\phi^*$ and  $\phi'$ and Lemma~\ref{lem:tphi}(ii), we have that both terms on the right-hand side converge pointwise to zero, which yields 
\[
\lim_{\epsilon\to 0^+}\phi^*\left(\left|\frac{\phi'(|u_{x_i}(x)|)}{|u_{x_i}(x)|}u_{x_i}(x) -\psi_{\epsilon}' (u_{x_i}(x))\right|\right)=0 
\quad\text{for a.e. } x\in B_{2r}.
\]
Moreover, by \eqref{equivalent},
\[
\phi^*\left(\left|\frac{\phi'(|u_{x_i}|)}{|u_{x_i}|}u_{x_i} -\psi_{\epsilon}' (u_{x_i})\right|\right) \le c \phi^*\big(\phi'(|u_{x_i}|)\big) + c \phi^*\big(\phi'(\epsilon+|u_{x_i}|)\big)\le c \phi(1+|u_{x_i}|),
\]
where the last function is integrable over $B_{2r}$.
Hence, by Lebesgue's dominated convergence theorem, 
\[
\lim_{\epsilon\to 0^+}\int_{B_{2r}}\phi^*\left(\left|\frac{\phi'(|u_{x_i}|)}{|u_{x_i}|}u_{x_i} -\psi_{\epsilon}' (u_{x_i})\right|\right) \,dx =0 
\quad \text{for all }\ i=1,2,\dots,n.
\]
This, together with Young's inequality and the fact that $\int_{B_{2r}}\phi(|u^\epsilon_{x_i} -u_{x_i}|)\,dx \le c M$ by \eqref{eq:energy_estimate1}, implies 
\[
I_2 \le c \sum_{i=1}^n\left\|\frac{\phi'(|u_{x_i}|)}{|u_{x_i}|}u_{x_i} -\psi_{\epsilon}' (u_{x_i})\right\|_{L^{\phi^*}(B_{2r})}\|u^\epsilon_{x_i} -u_{x_i}\|_{L^\phi(B_{2r})} \to 0 
\quad \text{as }\ \epsilon\to 0^+ .
\]

Therefore, we obtain 
\begin{equation}\label{eq:covergence_g_epsilon}
\lim_{\epsilon\to 0^+}\sum_{i=1}^n \int_{B_{2r}} \left(\psi_{\epsilon}' (u^\epsilon_{x_i}) -\psi_{\epsilon}' (u_{x_i})\right)(u^{\epsilon}_{x_i}-u_{x_i}) \, dx =0.
\end{equation}
We recall the following inequality (see e.g. \cite[Proposition 3.8(3)]{HasO22}):  
for any $\kappa\in(0,1)$,
\[
\psi_{\epsilon}(|u^{\epsilon}_{x_i}-u_{x_i}|) \le   \kappa \left(\psi_{\epsilon}(|u^{\epsilon}_{x_i}|) +\psi_{\epsilon}(|u_{x_i}|) \right) + \frac{c}{\kappa} \left(\frac{\psi_{\epsilon}'(|u^{\epsilon}_{x_i}|)}{|u^{\epsilon}_{x_i}|}u^{\epsilon}_{x_i}-\frac{\psi_{\epsilon}'(|u_{x_i}|)}{|u_{x_i}|}u_{x_i}\right)(u^{\epsilon}_{x_i}-u_{x_i}),
\] 
and observe from \eqref{def:g} and Lemma~\ref{lem:tphi}(ii) that 
\begin{equation}\label{eq:gepsilonphi}
 \phi(|t|)-(1+\epsilon)^{q-1}\phi(\epsilon)  \le  \psi_\epsilon(t)  \le (1+\epsilon)^{q-1} \phi (\epsilon+|t| ) 
 \quad \text{for all }\ t\in \R,
\end{equation}
which  implies that
\[
0\le \phi(|t|) \le   \psi_\epsilon(t) + (1+\epsilon)^{q-1}\phi(\epsilon)
\quad\text{and}\quad
0\le \psi_{\epsilon}(t) \le 2^{q-1}\phi(1+|t|)
\quad \text{for all }\ t\in \R.
\]
Using the previous results, the fact that $\frac{\psi_{\epsilon}'(|t|)}{|t|}t=\psi_{\epsilon}'(t)$, \eqref{eq:energy_estimate1} and \eqref{eq:covergence_g_epsilon}, we have that 
\[\begin{split}
\limsup_{\epsilon\to 0^+} \int_{B_{2r}}\phi(|\nabla u^{\epsilon} - \nabla u|)\,dx  
&\le c \limsup_{\epsilon\to 0^+} \sum_{i=1}^n \int_{B_{2r}}\phi(|u^{\epsilon}_{x_i}-u_{x_i}|)\,dx\\
 &\le \limsup_{\epsilon\to 0^+} \sum_{i=1}^n \int_{B_{2r}}\psi_{\epsilon}(|u^{\epsilon}_{x_i}-u_{x_i}|)\,dx \\
& \le c \kappa  \limsup_{\epsilon\to 0^+}  \sum_{i=1}^n \int_{B_{2r}}\left(\psi_{\epsilon}(|u^{\epsilon}_{x_i}|) +\psi_{\epsilon}(|u_{x_i}|) \right)\,dx\\
&\quad + \frac{c}{\kappa} \limsup_{\epsilon\to 0^+}   \sum_{i=1}^n \int_{B_{2r}} \left(\psi_{\epsilon}' (u^\epsilon_{x_i}) -\psi_{\epsilon}' (u_{x_i})\right)(u^{\epsilon}_{x_i}-u_{x_i}) \, dx\\
& \le c\kappa.
\end{split}\]
Consequently, since $\kappa\in(0,1)$ is arbitrary, we have  proven that 
\[
\lim_{\epsilon\to 0^+}\int_{B_{2r}}\phi(|\nabla u^{\epsilon}-\nabla u|)\,dx=0. \qedhere
\]
\end{proof}
Based on the previous result, we can provide estimates  directly for the regularized nondegenerate problems \eqref{PDE_regularized_BC}.
The first result is the  reverse H\"older type integrability which is a direct consequence of Lemma~\ref{lem:high} applied with $\psi_\epsilon$ in place of $\phi$, together with \eqref{eq:gepsilonphi}.

\begin{lemma}[reverse H\"older type integrability]\label{lem:high1}
Under the above setting, there exists $\sigma_0\in (0,1)$ and $c>0$  depending only on $n$, $p$, and $q$ such that 
\[
\begin{split}
\fint_{B_{r}} \psi_{\epsilon} (|\nabla u^\epsilon|)^{1+\sigma_0}\,dx
 \le  c \left( \fint_{B_{2r}}  \psi_\epsilon ( |\nabla u^\epsilon|) \,dx\right)^{1+\sigma_0} 
&\le  c \left( \fint_{B_{2r}}  \psi_\epsilon ( |\nabla U^\epsilon|) \,dx\right)^{1+\sigma_0}\\
& \le  c \left( \fint_{B_{2r}} \phi (| \nabla u|)\,dx+1\right)^{1+\sigma_0} .
\end{split}\]
\end{lemma}

The second result concerns the smoothness of the weak solution $u^\epsilon$ to \eqref{PDE_regularized_BC}.

\begin{proposition}\label{prop:u_epsilon}
Under the above setting, $u^\epsilon$ is $C^{\infty}$ in $B_{2r}$.
\end{proposition}
\begin{proof} 
Define $\Psi(\xi):=\sum_{k=1}^n \psi_\epsilon(\xi_k)$. We will show that the functional $I=\mathfrak F_{\psi_\epsilon}$ defined in \eqref{eq:regularizedenergy} satisfies the assumptions of \cite[Theorem 9.4]{Sta63} with $f=\Psi$, so that, since $\Psi(\xi)$ is smooth for $\xi=(\xi_1,\cdots, \xi_n)\in \R^n$, the theorem implies the smoothness of  $u^\epsilon$ in $B_{2r}$. Since the functional $\mathfrak F_{\psi_\epsilon}$ is autonomous, it suffices to verify Condition III in \cite[Eq. (9.1)]{Sta63}. For $\xi=(\xi_1,\cdots, \xi_n), z=(z_1,\cdots,z_n) \in \R^n$, by \eqref{eq:g''}, and Lemma~\ref{lem:tphi}, we have 
\[\begin{split}
\left( \Psi(\xi)\right)_{\xi_i \xi_j} z_i z_j&=
\left(\sum_{k=1}^n \psi_{\epsilon}(\xi_k)\right)_{\xi_i \xi_j} z_i z_j
 = \big(\psi_{\epsilon}'(\xi_i)\big)_{\xi_j} z_i z_j   = \psi_{\epsilon}''(\xi_i) \delta_{ij}  z_i z_j\\
& \ge \min\{p-1,1\}\frac{\tilde\phi'_{\epsilon}((\epsilon^2 + \xi_i^2)^{\frac12})}{(\epsilon^2 + \xi_i^2)^{\frac12}}  \delta_{ij}  z_i z_j\\
& \ge 
\begin{cases}
\min\{p-1,1\}\tilde\phi'_{\epsilon}(1) (\epsilon^2 + \xi_i^2)^{\frac{\min\{p, 2\}-2}2}  \delta_{ij}  z_i z_j,\quad \text{if }\ (\epsilon^2 + \xi_i^2)^{\frac12} \ge 1,\\
\min\{p-1,1\} \tilde\phi'_{\epsilon}(\epsilon)  \delta_{ij}  z_i z_j,\quad \text{if }\ (\epsilon^2 + \xi_i^2)^{\frac12} < 1.
\end{cases}
\end{split}\]
Therefore, it turns out that
\[\begin{split}
\sum_{j=1}^n \left(\Psi(\xi)\right)_{\xi_i \xi_j} z_i z_j
& \ge \begin{cases}
\min\{p-1,1\}\tilde\phi'_{\epsilon}(1) (1 + |\xi|^2)^{\frac{\min\{p, 2\}-2}2}  z_i^2 ,\quad \text{if }\ (\epsilon^2 + \xi_i^2)^{\frac12} \ge 1,\\
\min\{p-1,1\}  \tilde\phi'_{\epsilon}(\epsilon)   z_i^2,\quad \text{if }\ (\epsilon^2 + \xi_i^2)^{\frac12} < 1
\end{cases}\\
&\ge  \nu_\epsilon   (1 + |\xi|^2)^{\frac{\min\{p, 2\}-2}2}  z_i^2
\end{split}
\]
for some $\nu_\epsilon>0$, and so
\[\begin{split}
\sum_{i,j=1}^n \left(\Psi(\xi)\right)_{\xi_i \xi_j} z_i z_j
& \ge  \nu_\epsilon   (1 + |\xi|^2)^{\frac{\min\{p, 2\}-2}2}  |z|^2. \qedhere
\end{split}
\]

\end{proof}

\section{Caccioppoli-type inequalities}
Assume that $B_{2r}\Subset\Omega$. For simplicity, we consider the case $r=1$, which means $B_2\Subset \Omega$.
%\begin{equation}\label{assump:normal}
%B_2\Subset\Omega
%\quad \text{and}\quad
%\fint_{B_2} \phi(|\nabla u|) \,dx \le 1.
%\end{equation}
Fix $\epsilon\in (0,\epsilon_1)$ with $\epsilon_1:= \min\{1,\mathrm{dist}(B_2,\partial \Omega)\}$, and let $u^\epsilon \in W^{1,\phi}(B_2)$ be the weak solution to \eqref{PDE_regularized_BC} with $r=1$. Note that 
 by 
 %Lemma~\ref{lem:high1} and 
 Proposition~\ref{prop:u_epsilon}, $u^\epsilon$ is smooth.
%  and satisfies 
%$$
%\int_{B_1} \psi_\epsilon(|\nabla u^\epsilon|)^{1+\sigma_0} \,dx \le c.
%$$
Moreover, $u^{\epsilon}$  satisfies the equation
$$
\sum_{i=1}^n (\psi_{\epsilon}'(u^{\epsilon}_{x_i}) )_{x_i} = 0 \quad \text{pointwise in }\ B_2.
$$
In particular, for any $\zeta \in W^{1,1}_{0}(B_{2})$, by multiplying the above equation by $\zeta_{x_j}$ and integrating over $B_2$ and applying integration by parts, we obtain  
\begin{equation}\label{weakform2}
\sum_{i=1}^n \int_{B_{2}} \psi_{\epsilon}''(u^{\epsilon}_{x_i}) u^{\epsilon}_{x_ix_j} \,\zeta_{x_i} \, dx =0
\end{equation}
for all $j=1,2,\dots,n$. 

In this section, we establish Caccioppoli-type estimates that extend the results obtained for the $p$-power growth case in~\cite{BouBraLeoVer18} to the general Orlicz setting. 
For convenience, we will  write $u$ instead of $u^\epsilon$ for the remainder of this section.
%%%%%%%%%%%%%%%%%%%%%%%%%%%%%%
\begin{lemma}
Let $G : \R \rightarrow \R^+$ be a $C^1$-convex function. There exists a constant $c >0$ depending on $n, p,q$ such that for every $\eta \in C_0^{\infty}(B_{2})$ and every $j =1,2,\dots, n$, we have
\begin{equation}\label{ineq:Caccioppoli}
\sum_{i=1}^n \int_{B_{2}} \psi_{\epsilon}'' (u_{x_i}) |(G(u_{x_j}))_{x_i}|^2 \eta^2 \,dx 
\le c  \sum_{i=1}^n \int_{B_{2}} \psi_{\epsilon}'' (u_{x_i}) G(u_{x_j})^2 \eta_{x_i}^2 \,dx.
\end{equation}
\end{lemma}
\begin{proof}
Suppose that $G \in C^2$ for simplicity. 
Otherwise, a standard smoothing argument can be applied.
 Taking  $\zeta:=\eta^2 G(u_{x_j}) G'(u_{x_j})$ as a test function  in \eqref{weakform2}, we have 
 \[
 \begin{split}
 0
 &=\sum_{i=1}^n \int_{B_{2}} \psi_{\epsilon}'' (u_{x_i})   [2 \eta \eta_{x_i} G(u_{x_j}) G'(u_{x_j})u_{x_i x_j}  + \eta^2 G'(u_{x_j})^2 u_{x_i x_j}^2  + \eta^2 G(u_{x_j})G''(u_{x_j}) u^2_{x_i x_j}] \, dx\\
&\ge \sum_{i=1}^n \int_{B_{2}} \psi_{\epsilon}'' (u_{x_i})   [2 \eta \eta_{x_i} G(u_{x_j}) G'(u_{x_j})u_{x_i x_j}  + \eta^2 G'(u_{x_j})^2 u_{x_i x_j}^2  ] \, dx,
\end{split} 
\] 
where, in the last inequality, we have used the convexity of  $\psi_\epsilon$ and $G$, i.e., $\psi_{\epsilon}'', G''\ge 0$. Therefore, applying Young's inequality and using the fact that $G'(u_{x_j})u_{x_i x_j}= [G(u_{x_j})]_{x_i}$, we obtain \eqref{ineq:Caccioppoli}.
\end{proof}

\begin{lemma}[Weird Caccioppoli inequality] \label{lem:wCaccioppoli}
Suppose that $F, G :[0,\infty) \rightarrow [0, \infty) $ are nondecreasing continuous functions, and $G$ is $C^1$-convex. 
Let $\eta \in C^{\infty}_0(B_{2})$ and $0\le \theta \le 2$. 
Then we have that for every $k,j =1,2,\dots, n$, 
\begin{equation}\label{ineq:wCaccioppoli}
\begin{split}
&\sum_{i=1}^n \int_{B_{2}} \psi_{\epsilon}'' (u_{x_i}) u^2_{x_i x_j}  F(u^2_{x_j}) G(u^2_{x_k}) \eta^2 \,dx \\
&\le   c  \sum_{i=1}^n \int_{B_{2}} \psi_{\epsilon}'' (u_{x_i}) u_{x_j}^2 F(u^2_{x_j})G(u^2_{x_k})\eta_{x_i}^2\, dx \\
 & \qquad + c \bigg(\sum_{i=1}^n \int_{B_{2}} \psi_{\epsilon}'' (u_{x_i}) u^2_{x_i x_j} u^2_{x_j} F(u^2_{x_j})^2 G'(u^2_{x_k})^{\theta} \eta^2 \,dx\bigg)^{\frac12}\\
&\qquad \qquad \times \bigg( \sum_{i=1}^n \int_{B_{2}} \psi_{\epsilon}'' (u_{x_i}) |u_{x_k}|^{2\theta}   G(u^2_{x_k})^{2-\theta} \eta_{x_i}^2 \,dx\bigg)^{\frac12}
\end{split}
\end{equation}
for some $c=c(n, p, q)>0.$
\end{lemma}
\begin{proof}
By a standard approximation argument, we may assume that $F$ is $C^1$.
We take a function $\zeta := u_{x_j} F(u^2_{x_j})G(u^2_{x_k})\eta^2$ into \eqref{weakform2} to derive 
\[
\begin{split}
&\sum_{i=1}^n \int_{B_{2}} \psi_{\epsilon}'' (u_{x_i}) u^2_{x_i x_j} \big[ F(u^2_{x_j}) +2u^2_{x_j} F'(u^2_{x_j}) \big] G(u^2_{x_k}) \eta^2 \,dx \\
& =  -2  \sum_{i=1}^n \int_{B_{2}} \psi_{\epsilon}'' (u_{x_i}) u_{x_i x_j} u_{x_j} F(u^2_{x_j})G(u^2_{x_k})\eta \,\eta_{x_i}\, dx \\
& \qquad - 2 \sum_{i=1}^n \int_{B_{2}} \psi_{\epsilon}'' (u_{x_i}) u_{x_i x_j} u_{x_j} u_{x_i x_k} u_{x_k} F(u^2_{x_j}) G'(u^2_{x_k}) \eta^2 \,dx \\
& \le  \frac12  \sum_{i=1}^n \int_{B_{2}} \psi_{\epsilon}'' (u_{x_i}) u_{x_i x_j}^2 F(u^2_{x_j})G(u^2_{x_k})\eta^2\, dx \\
&\qquad +  8  \sum_{i=1}^n \int_{B_{2}} \psi_{\epsilon}'' (u_{x_i}) u_{x_j}^2 F(u^2_{x_j})G(u^2_{x_k})\eta_{x_i}^2\, dx \\
& \qquad - 2 \sum_{i=1}^n \int_{B_{2}} \psi_{\epsilon}'' (u_{x_i}) u_{x_i x_j} u_{x_j} u_{x_i x_k} u_{x_k} F(u^2_{x_j}) G'(u^2_{x_k}) \eta^2 \,dx,
\end{split}
\]
where we used Young's inequality in the last inequality. 
Then the first term on the right-hand side of the above inequality can be absorbed into the left-hand side, yielding
\begin{equation}\label{ineq:mid}
\begin{split}
&\frac12\sum_{i=1}^n \int_{B_{2}} \psi_{\epsilon}'' (u_{x_i}) u^2_{x_i x_j}  F(u^2_{x_j}) G(u^2_{x_k}) \eta^2 \,dx \\
&\le   8  \sum_{i=1}^n \int_{B_{2}} \psi_{\epsilon}'' (u_{x_i}) u_{x_j}^2 F(u^2_{x_j})G(u^2_{x_k})\eta_{x_i}^2\, dx \\
& \qquad - 2 \sum_{i=1}^n \int_{B_{2}} \psi_{\epsilon}'' (u_{x_i}) u_{x_i x_j} u_{x_j} u_{x_i x_k} u_{x_k} F(u^2_{x_j}) G'(u^2_{x_k}) \eta^2 \,dx.
\end{split}
\end{equation}

We now estimate the last term in \eqref{ineq:mid}.
Writing $G'(u^2_{x_k})=G'(u^2_{x_k})^{\frac{\theta}{2}}G'(u^2_{x_k})^{1-\frac{\theta}{2}}$ and applying Young's inequality, we obtain
\begin{equation}\label{ineq:mid_sub}
\begin{split}
&- 2 \sum_{i=1}^n \int_{B_{2}} \psi_{\epsilon}'' (u_{x_i}) u_{x_i x_j} u_{x_j} u_{x_i x_k} u_{x_k} F(u^2_{x_j}) G'(u^2_{x_k}) \eta^2 \,dx  \\
 &=  - 2 \sum_{i=1}^n \int_{B_{2}} \psi_{\epsilon}'' (u_{x_i}) u_{x_i x_j} u_{x_j} u_{x_i x_k} u_{x_k} F(u^2_{x_j}) G'(u^2_{x_k})^{\frac{\theta}2} G'(u^2_{x_k})^{1-\frac{\theta}2} \eta^2 \,dx\\
 & \le 2 \bigg(\sum_{i=1}^n \int_{B_{2}} \psi_{\epsilon}'' (u_{x_i}) u^2_{x_i x_j} u^2_{x_j} F(u^2_{x_j})^2 G'(u^2_{x_k})^{\theta} \eta^2 \,dx\bigg)^{\frac12}\\
&\qquad \times \bigg(\sum_{i=1}^n \int_{B_{2}} \psi_{\epsilon}'' (u_{x_i})  u^2_{x_i x_k} u^2_{x_k}  G'(u^2_{x_k})^{2-\theta} \eta^2 \,dx\bigg)^{\frac12}.
\end{split}
\end{equation} 
Here, we notice that
\[
\sum_{i=1}^n \int_{B_{2}} \psi_{\epsilon}'' (u_{x_i})  u^2_{x_i x_k} u^2_{x_k}  G'(u^2_{x_k})^{2-\theta} \eta^2 \,dx 
= \frac14 \sum_{i=1}^n \int_{B_{2}} \psi_{\epsilon}'' (u_{x_i}) \big| (G(u_{x_k}))_{x_i} \big|^2  \eta^2 \,dx,
\]
where
$G$ is the nonnegative convex $C^1$-function defined by 
\[
\widetilde G(t) = \int_{0}^{t^2} G'(\tau)^{1-\frac{\theta}2} \, d\tau.
\]
Then by applying \eqref{ineq:Caccioppoli} with $\widetilde G$ and $x_k$ instead of $G$ and $x_j$, respectively, we obtain
\[
\begin{split}
&\sum_{i=1}^n \int_{B_{2}} \psi_{\epsilon}'' (u_{x_i})  u^2_{x_i x_k} u^2_{x_k}  G'(u^2_{x_k})^{2-\theta} \eta^2 \,dx 
%= \frac14 \sum_{i=1}^n \int_{B_{2}} \psi_{\epsilon}'' (u_{x_i}) \big| (G(u_{x_k}))_{x_i} \big|^2  \eta^2 \,dx\\
%&\quad
\le c  \sum_{i=1}^n \int_{B_{2}} \psi_{\epsilon}'' (u_{x_i}) \widetilde G(u_{x_k})^2 \eta_{x_i}^2 \,dx. 
\end{split}
\]
Note that by Jensen's inequality with the fact that $G(0) \ge 0$,
\[
0\le \widetilde G(u_{x_k}) =  \int_{0}^{u_{x_k}^2} G'(\tau)^{1-\frac{\theta}2} \, d\tau \le |u_{x_k}|^{\theta} \bigg( \int_{0}^{u_{x_k}^2} G'(\tau) \, d\tau\bigg)^{1-\frac{\theta}2} \le |u_{x_k}|^\theta   G(u^2_{x_k})^{1-\frac{\theta}2}.
\]
Then we see that
\[
\begin{split}
&\sum_{i=1}^n \int_{B_{2}} \psi_{\epsilon}'' (u_{x_i})  u^2_{x_i x_k} u^2_{x_k}  G'(u^2_{x_k})^{2-\theta} \eta^2 \,dx 
%= \frac14 \sum_{i=1}^n \int_{B_{2}} \psi_{\epsilon}'' (u_{x_i}) \big| (G(u_{x_k}))_{x_i} \big|^2  \eta^2 \,dx\\
%&\quad
\le c  \sum_{i=1}^n \int_{B_{2}} \psi_{\epsilon}'' (u_{x_i}) |u_{x_k}|^{2\theta}   G(u^2_{x_k})^{2-\theta} \eta_{x_i}^2 \,dx. 
\end{split}
\]
Therefore, combining  \eqref{ineq:mid}   and \eqref{ineq:mid_sub} with the previous inequality, we derive 
the desired inequality \eqref{ineq:wCaccioppoli}.
\end{proof}

Let us consider a particular case of Lemma~\ref{lem:wCaccioppoli} by taking 
\begin{equation}\label{eq:FG}
F(t)=t^{s-1} \ \text{ and }\  G(t)=t^m 
\quad  \text {for }\  1\le s\le m \ \text{ and }\ t \ge 0.
\end{equation}
\begin{proposition}\label{prop:preCaccioppoli}
Let $\eta \in C^{\infty}_0(B_{2})$. Then for every
$k, j= 1,2, \dots, n$ and $1\le s \le m$, we have 
\begin{equation}\label{ineq:fCaccioppoli}
\begin{aligned}
&\sum_{i=1}^n \int_{B_{2}} \psi_{\epsilon}'' (u_{x_i}) u^2_{x_i x_j} |u_{x_j}|^{2s-2}|u_{x_k}|^{2m} \eta^2 \,dx \\
& \le\sum_{i=1}^n \int_{B_{2}} \psi_{\epsilon}'' (u_{x_i}) u^2_{x_i x_j} |u_{x_j}|^{4s-2}|u_{x_k}|^{2m-2s} \eta^2 \,dx + c  \sum_{i=1}^n \int_{B_{2}} \psi_{\epsilon}'' (u_{x_i}) |u_{x_j}|^{2s+2m} \eta_{x_i}^2 \,dx \\
& \qquad + 
c(m+1) \sum_{i=1}^n \int_{B_{2}} \psi_{\epsilon}'' (u_{x_i}) |u_{x_k}|^{2s+2m} \eta_{x_i}^2 \,dx
\end{aligned}
\end{equation}
for some $c=c(n, p, q)>0.$

\end{proposition}

\begin{proof}
We apply Lemma~\ref{lem:wCaccioppoli} with \eqref{eq:FG} and 
\[
\theta =
\begin{cases}
\frac{m-s}{m-1} \in [0,1] &\text{if}\ \ m > 1,\\
1&\text{if} \ \ m=1,
\end{cases}
\]
in order to discover that
\[
\begin{split}
&\sum_{i=1}^n \int_{B_{2}} \psi_{\epsilon}'' (u_{x_i}) u^2_{x_i x_j}  |u_{x_j}|^{2s-2} |u_{x_k}|^{2m} \eta^2 \,dx \\
& \le    c  \sum_{i=1}^n \int_{B_{2}} \psi_{\epsilon}'' (u_{x_i}) u_{x_j}^2 |u_{x_j}|^{2s-2}|u_{x_k}|^{2m} \eta_{x_i}^2\, dx \\
&\qquad  + c \bigg( m^{\theta}\sum_{i=1}^n \int_{B_{2}} \psi_{\epsilon}'' (u_{x_i}) u^2_{x_i x_j} |u_{x_j}|^{4s-2} |u_{x_k}|^{2m-2s} \eta^2 \,dx\bigg)^{\frac12} \\
&\qquad\qquad \times \bigg(\sum_{i=1}^n \int_{B_{2}} \psi_{\epsilon}'' (u_{x_i}) |u_{x_k}|^{2m+2s}  \eta_{x_i}^2 \,dx \bigg)^{\frac12}.
\end{split}
\]
Then by using Young's inequality,
% in the form $c (ab)^{\frac12} \le a+\frac{1}{4} c^2 b $, 
we obtain
\[
\begin{split}
&\sum_{i=1}^n \int_{B_{2}} \psi_{\epsilon}'' (u_{x_i}) u^2_{x_i x_j}  |u_{x_j}|^{2s-2} |u_{x_k}|^{2m} \eta^2 \,dx \\
%& \le    c  \sum_{i=1}^n \int_{B_{2}} \psi_{\epsilon}'' (u_{x_i})  |u_{x_j}|^{2s}|u_{x_k}|^{2m} |\nabla\eta|^2\, dx \\
%& + c \bigg( m^{\theta}\sum_{i=1}^n \int_{B_{2}} \psi_{\epsilon}'' (u_{x_i}) u^2_{x_i x_j} |u_{x_j}|^{4s-2} |u_{x_k}|^{2m-2s} \eta^2 \,dx\bigg)^{\frac12} 
%\bigg(\sum_{i=1}^n \int_{B_{2}} \psi_{\epsilon}'' (u_{x_i}) |u_{x_k}|^{2m+2s}  |\nabla\eta|^2 \,dx \bigg)^{\frac12} \\
&\le \sum_{i=1}^n \int_{B_{2}} \psi_{\epsilon}'' (u_{x_i}) u^2_{x_i x_j} |u_{x_j}|^{4s-2} |u_{x_k}|^{2m-2s} \eta^2 \,dx 
+c  \sum_{i=1}^n \int_{B_{2}} \psi_{\epsilon}'' (u_{x_i})  |u_{x_j}|^{2s}|u_{x_k}|^{2m} |\nabla\eta|^2\, dx \\
&  \qquad + c\, m^{\theta} \sum_{i=1}^n \int_{B_{2}} \psi_{\epsilon}'' (u_{x_i}) |u_{x_k}|^{2m+2s}  |\nabla\eta|^2 \,dx. 
\end{split}
\]
The desired estimates follow by applying Young's inequality with exponents $ \frac{2m+2s}{2s}, \frac{2m+2s}{2m}$ 
to the second term on the right-hand side, and by using the fact that $m^\theta \le m$ for the last term.
\end{proof}

%{\color{red}Change $q=> \bar{\ell}$.}

Finally, we have arrived at the following Caccioppoli type estimate for power functions of the gradient of $u$.
\begin{proposition}\label{prop:caccioppoli}
Let $\bar{\ell} = 2^{\ell_0}-1$ with  $\ell_0 \in \mathbb{N}$ and $\eta \in C^{\infty}_0(B_{2})$.
Then for every $k=1,2,\dots,n$, we have 
\begin{equation}\label{ineq:gCaccioppoli}
\begin{split}
&  \int_{B_{2}} \left| \nabla \left(|u_{x_k}|^{\bar{\ell}} \big[\psi_\epsilon( u_{x_k})\big]^{\frac12} \right)\right|^2 \eta^2 \,dx \\
& \le 
 c \, \bar{\ell}^{\,4}  \sum_{i, j =1}^n \int_{B_{2}} \psi_{\epsilon}'' (u_{x_i}) |u_{x_j}|^{2\bar{\ell}+2 }  \eta_{x_i}^2 \,dx + 
c\,\bar{\ell}^{\, 4}
 \sum_{i=1}^n \int_{B_{2}} \psi_{\epsilon}'' (u_{x_i}) |u_{x_k}|^{2\bar{\ell}+2 }\eta_{x_i}^2 \,dx,
\end{split}
\end{equation}
for some $c = c (n,p,q)>0$.
\end{proposition}

\begin{proof}
Let us define two families of indices $\{ s_\ell\}$ and $\{ m_\ell\}$ such that 
\[
s_\ell := 2^\ell
\ \text{ and }\ 
m_\ell := \bar{\ell}+1-2^\ell =\bar{\ell}+1 -s_{\ell}
\qquad \text{for }\  \ell \in \{ 0, \dots, \ell_0\}.
\]
Note that $s_0 =1,\ m_0=\bar{\ell}, \ s_{\ell_0} = 2^{\ell_0}, \ m_{\ell_0}=0$, $1\le s_\ell \le m_\ell$ for $\ell \in \{ 0, \dots, \ell_0-1\}$, and $2m_\ell -2s_\ell = 2m_{\ell+1}$.
Then inequality \eqref{ineq:fCaccioppoli} yields that 
\[\begin{split}
&\sum_{i=1}^n \int_{B_{2}} \psi_{\epsilon}'' (u_{x_i}) u^2_{x_i x_j} |u_{x_j}|^{2s_\ell -2}|u_{x_k}|^{2m_\ell } \eta^2 \,dx \\
& \le \sum_{i=1}^n \int_{B_{2}} \psi_{\epsilon}'' (u_{x_i}) u^2_{x_i x_j} |u_{x_j}|^{2s_{\ell+1} -2}|u_{x_k}|^{2m_{\ell+1} } \eta^2 \,dx + c  \sum_{i=1}^n \int_{B_{2}} \psi_{\epsilon}'' (u_{x_i}) |u_{x_j}|^{2\bar{\ell}+2 } \eta_{x_i}^2 \,dx \\
& \qquad + c(m_\ell+1) \sum_{i=1}^n \int_{B_{2}} \psi_{\epsilon}'' (u_{x_i}) |u_{x_k}|^{2\bar{\ell}+2 } \eta_{x_i}^2 \,dx
\end{split}
\]
for some constant $c>0$. Proceeding by iteration of the previous estimate from $\ell=0$ to $\ell=\ell_0-1$, we deduce that
\[\begin{split}
&\sum_{i=1}^n \int_{B_{2}} \psi_{\epsilon}'' (u_{x_i}) u^2_{x_i x_j} |u_{x_k}|^{2\bar{\ell} } \eta^2 \,dx \\
& \le 
\sum_{i=1}^n \int_{B_{2}} \psi_{\epsilon}'' (u_{x_i}) u^2_{x_i x_j} |u_{x_j}|^{2\bar{\ell}} \eta^2 \,dx 
+ c\, \ell_0  \sum_{i=1}^n \int_{B_{2}} \psi_{\epsilon}'' (u_{x_i}) |u_{x_j}|^{2\bar{\ell}+2 } \eta_{x_i}^2 \,dx \\
& \qquad +c \left((\ell_0-1)2^{\ell_0} +1\right)
 \sum_{i=1}^n \int_{B_{2}} \psi_{\epsilon}'' (u_{x_i}) |u_{x_k}|^{2\bar{\ell}+2 }  \eta_{x_i}^2 \,dx
\end{split}
\]
for a universal constant $c >0$. Applying the Caccioppoli type inequality \eqref{ineq:Caccioppoli} with 
$
G(t) = \frac{|t|^{\bar{\ell}+1}}{\bar{\ell}+1},
$
we obtain 
\[\begin{split}
\sum_{i=1}^n \int_{B_{2}} \psi_{\epsilon}'' (u_{x_i})u^2_{x_i x_j} |u_{x_j}|^{2\bar{\ell}} \eta^2 \,dx &=\sum_{i=1}^n \int_{B_{2}} \psi_{\epsilon}'' (u_{x_i}) \bigg|\left(\frac{|u_{x_j}|^{\bar{\ell}+1}}{\bar{\ell}+1}\right)_{x_i} \bigg|^2 \eta^2 \,dx \\
%&\le c  \sum_{i=1}^n \int_{B_{2}} \psi_{\epsilon}'' (u_{x_i}) \bigg|\frac{|u_{x_j}|^{\bar{\ell}+1}}{\bar{\ell}+1}\bigg|^2 |\eta_{x_i}|^2 \,dx 
& \le \frac{c}{(\bar{\ell}+1)^2}  \sum_{i=1}^n \int_{B_{2}} \psi_{\epsilon}'' (u_{x_i}) |u_{x_j}|^{2\bar{\ell}+2}\eta_{x_i}^2 \,dx.
\end{split}
\]
Combining this with the previous inequality, we obtain that
\[\begin{split}
&\sum_{i=1}^n \int_{B_{2}} \psi_{\epsilon}'' (u_{x_i}) u^2_{x_i x_j} |u_{x_k}|^{2\bar{\ell} } \eta^2 \,dx \\
& \le 
 c\, \bar{\ell}  \sum_{i=1}^n \int_{B_{2}} \psi_{\epsilon}'' (u_{x_i}) |u_{x_j}|^{2\bar{\ell}+2 } \eta_{x_i}^2 \,dx + 
c\,\bar{\ell}^{\, 2} \sum_{i=1}^n \int_{B_{2}} \psi_{\epsilon}'' (u_{x_i}) |u_{x_k}|^{2\bar{\ell}+2 } \eta_{x_i}^2 \,dx\\
& \qquad +%\sum_{i=1}^n \int_{B_{2}} \psi_{\epsilon}'' (u_{x_i}) u^2_{x_i x_j} |u_{x_j}|^{2q} \eta^2 \,dx\\
 \frac{c}{(\bar{\ell}+1)^2}  \sum_{i=1}^n \int_{B_{2}} \psi_{\epsilon}'' (u_{x_i}) |u_{x_j}|^{2\bar{\ell}+2} \eta_{x_i}^2 \,dx,
\end{split}
\]
which implies that 
\begin{equation}\label{est:afterCaccioppoli}
\begin{split}
&\sum_{i=1}^n \int_{B_{2}} \psi_{\epsilon}'' (u_{x_i}) u^2_{x_i x_j} |u_{x_k}|^{2\bar{\ell} } \eta^2 \,dx \\
& \le 
 c \, \bar{\ell}  \sum_{i=1}^n \int_{B_{2}} \psi_{\epsilon}'' (u_{x_i}) |u_{x_j}|^{2\bar{\ell}+2 } \eta_{x_i}^2 \,dx + 
c\, \bar{\ell}^{\, 2}
 \sum_{i=1}^n \int_{B_{2}} \psi_{\epsilon}'' (u_{x_i}) |u_{x_k}|^{2\bar{\ell}+2 }  \eta_{x_i}^2 \,dx.\\
\end{split}
\end{equation}

We note from \eqref{eq:g'g''}, which also implies $ \psi_\epsilon(t) \approx t \psi_{\epsilon}'(t)$, that 
\[\begin{split}
\left| \left(|u_{x_k}|^{\bar{\ell}} \big[\psi_\epsilon( u_{x_k})\big]^{\frac12} \right)_{x_j}\right|^2 
& \le 2 \, \bar{\ell}^{\, 2} |u_{x_k}|^{2\bar{\ell}-2} \psi_\epsilon( u_{x_k}) u_{x_k x_j}^2 +  |u_{x_k}|^{2\bar{\ell}} \big[\psi_\epsilon( u_{x_k})  \big]^{-1} | \psi_\epsilon'( u_{x_k})|^2 u_{x_k x_j}^2  \\
& \le  c \bar{\ell}^{\, 2} |u_{x_k}|^{2\bar{\ell}} \psi_\epsilon''( u_{x_k}) u_{x_k x_j}^2
 \end{split}\]
 for some constant $c = c(n, p,q)>0$.
Then  it follows that
 \[\begin{split}
 \int_{B_{2}} \left| \left(|u_{x_k}|^{\bar{\ell}} \big[\psi_\epsilon( u_{x_k})\big]^{\frac12} \right)_{x_j}\right|^2 \eta^2  \,dx 
&\le  c\, \bar{\ell}^{\, 2} \int_{B_{2}} \psi_{\epsilon}'' (u_{x_k}) u^2_{x_k x_j} |u_{x_k}|^{2\bar{\ell} } \eta^2 \,dx\\
&\le  c \, \bar{\ell}^{\, 2}\sum_{i=1}^n \int_{B_{2}} \psi_{\epsilon}'' (u_{x_i}) u^2_{x_i x_j} |u_{x_k}|^{2\bar{\ell} } \eta^2 \,dx\\
& \le 
 c\, \bar{\ell}^{\, 4}  \sum_{i=1}^n \int_{B_{2}} \psi_{\epsilon}'' (u_{x_i}) |u_{x_j}|^{2\bar{\ell}+2 } \eta_{x_i}^2 \,dx \\
 &\qquad + 
c\, \bar{\ell}^{\, 4}
 \sum_{i=1}^n \int_{B_{2}} \psi_{\epsilon}'' (u_{x_i}) |u_{x_k}|^{2\bar{\ell}+2 } \eta_{x_i}^2 \,dx,
\end{split}
\]
where the last inequality follows from   
\eqref{est:afterCaccioppoli}. Therefore, after summing over $j =1,\dots, n$, we obtain \eqref{ineq:gCaccioppoli}.
\end{proof}

\section{Proof of Main Theorem}
%To prove our main theorem, Theorem~\ref{thm:main}, it suffices to derive the uniform Lipschitz estimate of Proposition~\ref{prop:unifLip} below.
To prove our main theorem, Theorem~\ref{thm:main}, we first establish  the uniform Lipschitz estimates for the weak solution $u^\epsilon$ to the regularized problem \eqref{PDE_regularized_BC} under the normalization assumption \eqref{assump:normal}.
%under the assumption \eqref{assump:normal}. 
Furthermore, as in the theorem, 
we assume that 
\begin{equation}\label{eq:pq}
2\le p \le q.
\end{equation}

\begin{proposition}[Uniform Lipschitz estimate]\label{prop:unifLip}
Suppose  \eqref{eq:pq} holds and
\begin{equation}\label{assump:normal}
B_2\Subset\Omega
\quad \text{and}\quad
\fint_{B_2} \phi(|\nabla u|) \,dx \le 1.
\end{equation} 
Let $u^\epsilon$ be the weak solution to \eqref{PDE_regularized_BC} with $r=1$.
Then
\begin{equation}\label{est:unifLip}
\|\nabla u^\epsilon \|_{L^{\infty}(B_{1/2})}  \le c
\end{equation}
%$B_r \subset B_{2r} \Subset \Omega$ 
%and $\sigma>0$ depending on $n$ and $p$. 
for some $c=c(n,p,q)>0$.
\end{proposition}

\begin{proof} %As in the previous section, we simply write $u=u^\epsilon$.

{\bf Step 1 (a first iterative scheme).}
Applying \eqref{ineq:gCaccioppoli}, we have that
\[
\begin{split}
& \int_{B_{2}}  \left| \nabla \left(|u^{\epsilon}_{x_k}|^{\bar{\ell}} \big[\psi_\epsilon( u^{\epsilon}_{x_k})\big]^{\frac12} \eta \right)\right|^2  \,dx  \\
&\le 2 \int_{B_{2}}  \left| \nabla \left(|u^{\epsilon}_{x_k}|^{\bar{\ell}} \big[\psi_\epsilon( u^{\epsilon}_{x_k})\big]^{\frac12} \right)\right|^2  \eta^2  \,dx +2 \int_{B_{2}} |u^{\epsilon}_{x_k}|^{2\bar{\ell}} \psi_\epsilon( u^{\epsilon}_{x_k}) |\nabla\eta|^2 \,dx  \\
& \le 
 c \, \bar{\ell}^{\,4}  \sum_{i, j =1}^n \int_{B_{2}} \psi_{\epsilon}'' (u^{\epsilon}_{x_i}) |u^{\epsilon}_{x_j}|^{2\bar{\ell}+2 } |\nabla \eta|^2 \,dx + 
c\, {\bar{\ell}}^{\, 4}
 \sum_{i=1}^n \int_{B_{2}} \psi_{\epsilon}'' (u^{\epsilon}_{x_i}) |u^{\epsilon}_{x_k}|^{2\bar{\ell}+2 } |\nabla \eta|^2 \,dx\\
 & \qquad +2 \int_{B_{2}} |u^{\epsilon}_{x_k}|^{2\bar{\ell}} \psi_\epsilon( u^{\epsilon}_{x_k}) |\nabla\eta|^2 \,dx ,
\end{split}
\]
from which, by the application of Sobolev inequality we derive that
\[
\begin{split}
& \bigg( \int_{B_{2}} \left[ |u^{\epsilon}_{x_k}|^{2 \bar{\ell}} \psi_\epsilon( u^{\epsilon}_{x_k})\eta^2  \right]^{\frac{2^*}{2}} \,dx \bigg)^{\frac{2}{2^*}}\\
& \le 
 c \bar{\ell}^4  \sum_{i, j =1}^n \int_{B_{2}} \psi_{\epsilon}'' (u^{\epsilon}_{x_i}) |u^{\epsilon}_{x_j}|^{2\bar{\ell}+2 } |\nabla \eta|^2 \,dx + 
c\bar{\ell}^4
 \sum_{i=1}^n \int_{B_{2}} \psi_{\epsilon}'' (u^{\epsilon}_{x_i}) |u^{\epsilon}_{x_k}|^{2\bar{\ell}+2 } |\nabla \eta|^2 \,dx\\
 & \qquad +c \int_{B_{2}}  |u^{\epsilon}_{x_k}|^{2\bar{\ell}} \psi_\epsilon( u^{\epsilon}_{x_k}) |\nabla\eta|^2 \,dx.
\end{split}
\]
Summing over $k =1,\dots, n$, we obtain  
\[
\begin{split}
& \bigg( \int_{B_{2}}  \left[ \sum_{k=1}^n  |u^{\epsilon}_{x_k}|^{2 \bar{\ell}} \psi_\epsilon( u^{\epsilon}_{x_k})  \right]^{\frac{2^*}{2}} \eta^{2^*} \,dx \bigg)^{\frac{2}{2^*}}\\
&\quad \le c \bar{\ell}^4 \sum_{i,k=1}^n \int_{B_{2}} \psi_{\epsilon}'' (u^{\epsilon}_{x_i}) |u^{\epsilon}_{x_k}|^{2\bar{\ell}+2 } |\nabla \eta|^2 \,dx +c \int_{B_{2}} \left[\sum_{k=1}^n  |u^{\epsilon}_{x_k}|^{2 \bar{\ell}} \psi_\epsilon( u^{\epsilon}_{x_k})\right] |\nabla\eta|^2  \,dx.
  \end{split}
\]
Now we set 
\[
\mathcal{U}(x):= \max_{k=1,\dots, n}  |u^{\epsilon}_{x_k}|.
\]
Since $p\ge 2$, we get
\[
\bigg( \int_{B_{2}}  \left[  |\mathcal{U}|^{2 \bar{\ell}} \psi_\epsilon( \mathcal{U})  \right]^{\frac{2^*}{2}} \eta^{2^*} \,dx \bigg)^{\frac{2}{2^*}} 
\le c\, \bar{\ell}^{\, 4} \int_{B_{2}} \left[  |\mathcal{U}|^{2 \bar{\ell}} \psi_\epsilon( \mathcal{U} )\right] |\nabla\eta|^2  \,dx.
\]

For every pair of radii $1/2\le s_1<s_2 \le 2,$
we take a standard cut-off function $\eta \in C^\infty_0(B_{s_2})$ such that $\eta \equiv 1$ on $B_{s_1}$, $0\le \eta \le 1$, and $\| \nabla \eta \|_{\infty} \le \frac{c}{s_2-s_1}$ into the previous inequality, in order to get 
\[
\bigg( \int_{B_{s_1}}  \left[  \mathcal{U}^{2 \bar{\ell}} \psi_\epsilon( \mathcal{U})  \right]^{\frac{2^*}{2}} \,dx \bigg)^{\frac{2}{2^*}} 
\le \frac{c\, \bar{\ell}^{\, 4} }{(s_2-s_1)^2}\int_{B_{s_2}} \mathcal{U}^{2 \bar{\ell}} \psi_\epsilon( \mathcal{U} )  \,dx.
\]

Setting the sequence of exponents 
\[
\gamma_j = 2^{j+2}-2, \   \  \ j \in \mathbb{N}
\]
and taking $\bar{\ell} = 2^{j+1}-1 = \frac{\gamma_j}2$, we then obtain 
\begin{equation}\label{ineq:gamma_j}
\bigg( \int_{B_{s_1}}  \left[  \mathcal{U}^{\gamma_j} \psi_\epsilon( \mathcal{U})  \right]^{\frac{2^*}{2}} \,dx \bigg)^{\frac{2}{2^*}} 
\le \frac{c\, \bar{\ell}^{\, 4} }{(s_2-s_1)^2}\int_{B_{s_2}} \mathcal{U}^{\gamma_j} \psi_\epsilon( \mathcal{U} )  \,dx.
\end{equation}

{\bf Step 2 (filling the gaps).}
We now observe that 
\[
\gamma_{j -1}< \gamma_j < \frac{2^* \gamma_j}{2},  \ \ \text{ for every } j \in \mathbb{N}.%\backslash \{0\}.
\]
By the elementary interpolation inequality for Lebesgue spaces with
\[
\frac{1}{\gamma_j} = \frac{\tau_j }{\gamma_{j-1}} + \frac{(1-\tau_j)}{2^* \gamma_j/2}, \text{ where } \tau_j := \frac{\frac{2^*}{2} -1}{\frac{2^*\gamma_j}{2\gamma_{j-1}}-1} \in (0,1),
\]
we deduce that
\[
\begin{split}
\int_{B_{s_1}}  \mathcal{U}^{\gamma_j} \psi_\epsilon( \mathcal{U})  \,dx
&=\int_{B_{s_1}} \big[ \mathcal{U}^{\gamma_j} \psi_\epsilon( \mathcal{U}) \big]^{\tau_j + (1-\tau_j)} \,dx\\
& \le \bigg( \int_{B_{s_1}} \big[ \mathcal{U}^{\gamma_j} \psi_\epsilon( \mathcal{U}) \big]^{\frac{\gamma_{j-1}}{\gamma_j}} \,dx\bigg)^{\frac{\tau_j \gamma_j}{\gamma_{j-1}}}
 \bigg( \int_{B_{s_1}} \big[\mathcal{U}^{\gamma_j} \psi_\epsilon( \mathcal{U}) \big]^{\frac{2^*}{2}}\,dx \bigg)^{\frac{2(1-\tau_j)}{2^*}}.
\end{split}
\]
Applying \eqref{ineq:gamma_j} to the last term, we obtain  
\[
\begin{split}
&\int_{B_{s_1}} \mathcal{U}^{\gamma_j} \psi_\epsilon( \mathcal{U})  \,dx
  \le \bigg( \int_{B_{s_1}} \big[\mathcal{U}^{\gamma_j} \psi_\epsilon( \mathcal{U}) \big]^{\frac{\gamma_{j-1}}{\gamma_j}} \,dx\bigg)^{\frac{\tau_j \gamma_j}{\gamma_{j-1}}}
 \bigg(  \frac{c \gamma_j^4}{(s_2-s_1)^2} \int_{B_{s_2}} \mathcal{U}^{\gamma_j} \psi_\epsilon( \mathcal{U}) \,dx \bigg)^{1-\tau_j}\\
&\quad = \bigg[\bigg(\frac{c 2^{4j}}{(s_2-s_1)^2} \bigg)^{\frac{1-\tau_j}{\tau_j}}\bigg( \int_{B_{s_1}} \big[\mathcal{U}^{\gamma_j} \psi_\epsilon( \mathcal{U}) \big]^{\frac{\gamma_{j-1}}{\gamma_j}} \,dx\bigg)^{\frac{\gamma_j}{\gamma_{j-1}}}\bigg]^{\tau_j}\bigg( \int_{B_{s_2}} \mathcal{U}^{\gamma_j} \psi_\epsilon( \mathcal{U})  \,dx\bigg)^{1-\tau_j}.
\end{split}
\]
Here, since the sequence $\{ \tau_j\}_{j\ge 1}$ is decreasing, we have
\[
\frac{2\cdot 2^* -4}{ 3\cdot 2^* -4} =\tau_1\ge \tau_j > \lim_{m\rightarrow \infty} \tau_m = \frac{2^* -2}{2(2^* -1)}: = \bar \tau \qquad \text{for every }\ j \in \mathbb{N},
\]
which yields that
\[
\frac{1-\tau_j}{\tau_j} \le \frac{1-\bar \tau}{\bar \tau}=: \beta.
\]
Then  
%we see 
%\[
%\bigg(\frac{c \gamma_j^4}{(s_2-s_1)^2} \bigg)^{\frac{1-\tau_j}{\tau_j}} \le \bigg(\frac{c 2^{4j}}{(s_2-s_1)^2} \bigg)^{\frac{1-\tau_j}{\tau_j}} \le \bigg(\frac{c 2^{4j}}{(s_2-s_1)^2} \bigg)^{\beta},
%\]
%and so 
using Young's inequality, we derive that for every $1/2\le  r \le s_1< s_2 \le R \le 2$, 
\[
\begin{split}
&\int_{B_{s_1}} \mathcal{U}^{\gamma_j} \psi_\epsilon( \mathcal{U})  \,dx\\
&\ \ \le \bigg[\bigg(\frac{c 2^{4j}}{(s_2-s_1)^2} \bigg)^{\beta}\bigg( \int_{B_{s_1}} \big[ \mathcal{U}^{\gamma_j} \psi_\epsilon( \mathcal{U}) \big]^{\frac{\gamma_{j-1}}{\gamma_j}} \,dx\bigg)^{\frac{\gamma_j}{\gamma_{j-1}}}\bigg]^{\tau_j}
\bigg( \int_{B_{s_2}} \mathcal{U}^{\gamma_j} \psi_\epsilon( \mathcal{U})  \,dx\bigg)^{1-\tau_j}\\
& \ \ \le \frac{1}{2}\int_{B_{s_2}} \mathcal{U}^{\gamma_j} \psi_\epsilon( \mathcal{U})  \,dx + c \frac{2^{4\beta j}}{(s_2-s_1)^{2\beta}}\bigg( \int_{B_{R}} \big[ \mathcal{U}^{\gamma_j} \psi_\epsilon( \mathcal{U}) \big]^{\frac{\gamma_{j-1}}{\gamma_j}} \,dx\bigg)^{\frac{\gamma_j}{\gamma_{j-1}}}.
\end{split}
\]
By applying Lemma~\ref{lem:tec} with $\theta:=1/2, \ \alpha_0 : = 2\beta$ and $Z(t)=\int_{B_{t}} \mathcal{U}^{\gamma_j} \psi_\epsilon( \mathcal{U})  \,dx$,
we conclude that 
\begin{equation}\label{est:r2r}
\int_{B_{r}} \mathcal{U}^{\gamma_j} \psi_\epsilon( \mathcal{U})  \,dx  \le c \frac{2^{4j \beta}}{(R-r)^{2\beta}} \bigg( \int_{B_{R}}  \mathcal{U}^{\gamma_{j-1}} \psi_\epsilon( \mathcal{U})  +1 \,dx \bigg)^{\frac{\gamma_j}{\gamma_{j-1}}}
\end{equation}
for some $c = c(n,p,q)>1$.

{\bf Step 3 (Moser's iteration).}
The foregoing estimate will now be iterated along a family of shrinking balls.
Let us fix $1/2\le r<R\le 2$, and consider the sequence 
\[
R_j = r+ \frac{R-r}{2^{j-1}}, \ \ j \in \mathbb{N}\backslash \{0\}.
\]
We apply \eqref{est:r2r} with $R_{j+1} < R_j$, instead of $r<R$, to discover 
\begin{equation}\label{est:R_j}
\begin{split}
\int_{B_{R_{j+1}}}  \mathcal{U}^{\gamma_j} \psi_\epsilon( \mathcal{U})   \,dx
& \le c \frac{2^{4j \beta}}{(R_j-R_{j+1})^{2\beta}}  \bigg( \int_{B_{R_j}}  \mathcal{U}^{\gamma_{j-1}} \psi_\epsilon( \mathcal{U}) +1\,dx\bigg)^{\frac{\gamma_j}{\gamma_{j-1}}}\\
&\le   c \frac{2^{6j \beta}}{(R-r)^{2\beta}} \bigg( \int_{B_{R_j}}  \mathcal{U}^{\gamma_{j-1}} \psi_\epsilon( \mathcal{U}) +1 \,dx\bigg)^{\frac{\gamma_j}{\gamma_{j-1}}}
\end{split}
\end{equation}
for some constant $ c=c(n,p,q)>1$.

We write 
\[
Y_{j} :=  \int_{B_{R_j}}   \mathcal{U}^{\gamma_{j-1}} \psi_\epsilon( \mathcal{U}) \,dx,
\]
and then \eqref{est:R_j} can be rewritten as 
\[
Y_{j+1}
\le  c \frac{2^{6j \beta}}{(R-r)^{2\beta}} \big( Y_{j} + |B_{R_j}|\big)^{\frac{\gamma_j}{\gamma_{j-1}}}\le c \frac{2^{6j \beta}}{(R-r)^{2\beta}} \big( Y_{j} + 1\big)^{\frac{\gamma_j}{\gamma_{j-1}}}.
\]
Here, we have used that $R \le 2$. 
By iterating the previous estimate starting from $j=1$, 
it follows that
\[
Y_{m+1} \le  ( c 2^{6 \beta}(R-r)^{-2\beta})^{\sum_{j=0}^{m-1} (m-j)\frac{\gamma_m}{\gamma_{m-j}}} \bigg( Y_{1} + 1\bigg)^{\frac{\gamma_m}{\gamma_{0}}}.
\]

Now we take the power $1/\gamma_m$ on both sides to get
\[
Y_{m+1}^{\frac1{\gamma_m}} \le  ( c 2^{6 \beta}(R-r)^{-2\beta})^{\sum_{j=0}^{m-1} \frac{m-j}{\gamma_{m-j}}} \big( Y_{1} + 1\big)^{\frac{1}{\gamma_{0}}} = ( c 2^{6 \beta}(R-r)^{-2\beta})^{\sum_{j=1}^{m} \frac{j}{\gamma_{j}}} \big( Y_{1} + 1\big)^{\frac{1}{\gamma_{0}}}.
\]
Setting 
\[
d\nu(x): =  \psi_\epsilon (\mathcal{U})  \,dx,
\]
we derive 
\[
\begin{split}
 \| \mathcal{U}\|_{L^{\infty}(B_r)} &= \| \mathcal{U}\|_{L^{\infty}(B_r ; d\nu)}= \lim_{m\rightarrow \infty } \left(\int_{B_{R_{m+1}}}  \mathcal{U}^{\gamma_{m}} \, d\nu \right)^{\frac1{\gamma_m}}
% \\ & = \lim_{m\rightarrow \infty } \left(\int_{B_{R_{m+1}}}  \mathcal{U}^{\gamma_{m}} \bigg[ \mathcal{U}^{2} \frac{\phi'((\epsilon^2 + \mathcal{U}^2)^{\frac12})}{(\epsilon^2 + \mathcal{U}^2)^{\frac12}} \bigg] \,dx\right)^{\frac1{\gamma_m}}\\&
= \lim_{m\rightarrow \infty }Y_{m+1}^{\frac1{\gamma_m}} \\
 & \le  ( c 2^{6 \beta}(R-r)^{-2\beta})^{\sum_{j=1}^{\infty} \frac{j}{\gamma_{j}}} \bigg( Y_{1} + 1\bigg)^{\frac{1}{\gamma_{0}}} \\
 & \le c (R-r)^{-\beta _1}\bigg( Y_{1} + 1\bigg)^{\frac{1}{2}} 
 = \frac{c}{(R-r)^{\beta _1}}\bigg( \int_{B_{R}}  \mathcal{U}^{2}\psi_\epsilon(\mathcal{U})\,dx + 1\bigg)^{\frac{1}{2}}
\end{split}
\]
for some $c = c(n,p,q)>1$ and $\beta _1 = \beta _1(n,p,q)>0$,
since \[
\sum_{j=1}^{\infty} \frac{j}{\gamma_{j}} = \sum_{j=1}^{\infty} \frac{j}{2^{j+2}-2}< \infty.
\]
Recalling the definition of $\mathcal{U}$, we conclude 
\[
\begin{split}
 \| \nabla u^{\epsilon} \|_{L^{\infty}(B_r)} 
 & \le \frac{c}{(R-r)^{\beta _1}}\bigg( \int_{B_{R}} |\nabla u^{\epsilon}|^{2}\psi_{\epsilon}( |\nabla u^{\epsilon}|)\,dx + 1\bigg)^{\frac{1}{2}}.
\end{split}
\]

{\bf Step 4: $L^{\infty}-L^{\psi_\epsilon}$ estimate.}
We recall the reverse H\"older inequality in Lemma~\ref{lem:high1}:
 \begin{equation}\label{assump:normalg}
\int_{B_1} \psi_\epsilon(|\nabla u^{\epsilon}|)^{1+\sigma} \,dx\le c \left(\int_{B_2} \psi_\epsilon(|\nabla u|) \,dx\right)^{1+\sigma} \le c
\end{equation}
for some small $\sigma>0$, where we have used \eqref{assump:normal}  in the last inequality. 
We may assume that $p\sigma<2$. By applying the previous estimate, we have that
for every $\frac12\le s_1<s_2\le 1$,
\[
\begin{split}
 \| \nabla u^{\epsilon} \|_{L^{\infty}(B_{s_1})} 
 & \le \frac{c}{(s_2-s_1)^{\beta _1}}\bigg(\int_{B_{s_2}}  |\nabla u^{\epsilon}|^{2}\psi_{\epsilon}( |\nabla u^{\epsilon}|)\,dx\bigg)^{\frac12} + \frac{c}{(s_2-s_1)^{\beta _1}}.
\end{split}
\]

We note from \eqref{eq:g'g''} with $p\ge 2$ that  
\[
|\nabla u^{\epsilon}|^{p\sigma}  
\le |\nabla u^{\epsilon}|^{p\sigma} |\nabla u^{\epsilon}|^{-p\sigma} \psi_\epsilon (|\nabla u^{\epsilon}|)^{\sigma} = \psi_\epsilon (|\nabla u^{\epsilon}|)^{\sigma} 
\quad \text{for } \ |\nabla u^{\epsilon} | > 1,
\]
which implies 
\[
|\nabla u^{\epsilon} |^{p\sigma} \psi_\epsilon(|\nabla u^{\epsilon}|) \le \psi_\epsilon(|\nabla u^{\epsilon}|)^{1+\sigma} + c .
\] 
Then by   \eqref{assump:normalg} and Young's inequality,
we derive  
\[
\begin{split}
 \| \nabla u^{\epsilon} \|_{L^{\infty}(B_{s_1})} 
 & \le \frac{c}{(s_2-s_1)^{\beta _1}}\bigg( \int_{B_{s_2}} |\nabla u^{\epsilon}|^{2}\psi_{\epsilon}( |\nabla u^{\epsilon}|)\,dx\bigg)^{\frac12} + \frac{c}{(s_2-s_1)^{\beta _1}}\\
 &\le   \frac{c}{(s_2-s_1)^{\beta _1}}\bigg( \int_{B_{s_2}} |\nabla u^{\epsilon}|^{p\sigma} \psi_\epsilon( |\nabla u^{\epsilon}|) \,dx\bigg)^{\frac12} \| \nabla u^{\epsilon} \|^{1-\frac{p\sigma}2}_{L^{\infty}(B_{s_2})}+ \frac{c}{(s_2-s_1)^{\beta _1}}\\
 &\le \frac{c}{(s_2-s_1)^{\beta _1}}\bigg( \int_{B_{s_2}}\psi_\epsilon(|\nabla u^{\epsilon}|)^{1+\sigma}  \,dx + 1\bigg)^{\frac12} \| \nabla u^{\epsilon} \|^{1-\frac{p\sigma}2}_{L^{\infty}(B_{s_2})}+ \frac{c}{(s_2-s_1)^{\beta _1}}\\
 &\le   \frac{c}{(s_2-s_1)^{\beta _1}} \| \nabla u^{\epsilon} \|^{1-\frac{p\sigma}2}_{L^{\infty}(B_{s_2})}+ \frac{c}{(s_2-s_1)^{\beta _1}}\\
  &\le  \frac12 \| \nabla u^{\epsilon} \|_{L^{\infty}(B_{s_2})} +\frac{c}{(s_2-s_1)^{\beta_2} }
 \end{split}
\]
for some $\beta_2>0$.
Finally, applying Lemma~\ref{lem:tec} with $\theta:=\frac{1}{2}, \ \alpha_0 : = \beta_2$, and 
$
Z(t):= \| \nabla u^{\epsilon} \|_{L^{\infty}(B_t)},
$
we conclude that 
\[
\| \nabla u^{\epsilon} \|_{L^{\infty}(B_{1/2})}  \le  c. \qedhere
\]
\end{proof}

\begin{lemma}(cf. \cite[Lemma~6.11]{Giusti_book})\label{lem:tec}
Let $0<r<R$ and let $Z : [r,R] \rightarrow [0,\infty)$ be a bounded function. Assume that for $r \le t < s \le R$, 
\[
Z(t) \le \frac{A}{(s-t)^{\alpha_0}} +  B + \theta Z(s),
\]
with $A, B \ge 0, \ \alpha_0 >0$ and $0\le \theta <1$. Then we have 
\[
Z(r) \le \bigg( \frac{1}{(1-\lambda)^{\alpha_0}}\frac{\lambda^{\alpha_0}}{\lambda^{\alpha_0}-\theta}\bigg) \bigg[\frac{A}{(R-r)^{\alpha_0}}  +  B \bigg],
\]
where $\lambda$ is any number such that 
$
\theta^{\frac1{\alpha_0}} < \lambda <1.
$
\end{lemma}

Finally, the main theorem follows from the previous proposition, using the strong convergence of $\{u^\epsilon\}$ to $u$ and a standard scaling argument.

\begin{proof}[Proof of Theorem~\ref{thm:main}]
Fix $B_{2r} \Subset \Omega$. Without loss of generality, we assume that the center of the ball $B_{2r}$ is the origin. We define $\bar{\phi}(t) := \frac{\phi(\lambda t )}{\phi(\lambda)}$ and $\bar{u}(x) := \frac{u(rx)}{\lambda r}$, where 
\[
\lambda := \phi^{-1} \bigg( \fint_{B_{2r}} \phi(|\nabla u|)\,dx \bigg).
\]
Then we observe that $\bar{\phi}$ also satisfies Assumption~\ref{assump:growth2} (with $\phi$ replaced by $\bar{\phi}$), and that $\bar{u}$ is a minimizer of the functional
\[
\mathfrak{F}_{\bar{\phi}}(\bar{u} ; B_2) := \sum_{i=1}^n \int_{B_2} \bar{\phi}(|\bar{u}_{x_i}|)\,dx.
\]
Consequently, it is a weak solution to
\[
\sum_{i=1}^n \left(\frac{\bar{\phi}' (|\bar{u}_{x_i}|)}{|\bar{u}_{x_i}|}\bar{u}_{x_i}\right)_{x_i} = 0 \quad \text{in } B_2,
\]
and we have
\[
\fint_{B_{2}} \bar{\phi}(|\nabla \bar{u}|)\,dx = 1.
\]

Let $\epsilon \in (0,\epsilon_1)$, where $\epsilon_1 = \min\{1, \mathrm{dist}(B_{2},\partial\tilde{\Omega})\}$ and $\tilde{\Omega} := \{x \in \mathbb{R}^n: rx \in \Omega\}$. We consider the minimizer $\bar{u}^\epsilon$ of the corresponding regularized variational problem 
\[
\mathfrak{F}_{\bar{\psi}_\epsilon}(v ; B_{2}) := \sum_{i=1}^n \int_{B_{2}} \bar{\psi}_\epsilon ( v_{x_i} ) \, dx, \quad v \in \bar{U}^{\epsilon} + W^{1,\bar{\phi}}_0(B_{2}),
\] 
where $\bar{\psi}_\epsilon$ is defined as in \eqref{def:g} with $\phi$ replaced by $\bar{\phi}$, and $\bar{U}^\epsilon(x) := \bar{u} * m_\epsilon (x)$. Then by 
Proposition~\ref{prop:unifLip}, $
\|\nabla \bar{u}^\epsilon \|_{L^\infty(B_{1/2})} \le c
$ for every small $\epsilon>0$, where the constant $c>0$ is independent of $\epsilon$,
and hence by the strong convergence of $\nabla \bar{u}^\epsilon$ to $\nabla \bar{u}$ in $L^{\bar{\phi}}(B_{2},\mathbb{R}^n)$ (Proposition~\ref{prop:strongconv}), and the uniform estimate \eqref{est:unifLip}, we obtain 
\[
\|\nabla \bar{u} \|_{L^\infty(B_{1/2})} \le c
\] 
for some constant $c = c(n,p,q) > 0$. Therefore, the desired estimate \eqref{estimate:Lip} follows by scaling back to the original variables.
\end{proof}

%%%%%%%%%%%%%%%%%%%%%%%%%%%%%%%%%%%%%%%%%%%%%%%%%%%%%%%%%%%%%%%%%%%%%%%%%%%%%%%%

\section*{Acknowledgment}
M. Lee was supported by the National Research Foundation of Korea(NRF) grant funded by the Korea government(MSIT) (RS-2025-23525636). 
J. Ok was supported by the National Research Foundation of Korea(NRF) grant funded by the Korea government(MSIT)
(NRF-2022R1C1C1004523). J. Ok and B. Stroffolini were funded by the University of Naples Federico II through  the International Co-operation between University of Naples Federico II and Sogang University. The research of B. Stroffolini is part of the project
 “Geometric Evolution Problems and Shape
Optimizations”, PRIN  Project 2022E9CF89. 

\subsection*{Data availability} Data sharing is not applicable to this article as obviously no datasets were generated or
analyzed during the current study.

\subsection*{Conflict of Interest} The authors declare no conflict of interest.
%%%%%%%%%%%%%%%%%%%%%%%%%%%%%%%%%%%%%%%%%%%%%%%%%%%%%%%%%%%%%%%%%%%%%%%%%%%%%%%%
\bibliographystyle{amsplain}

\end{document}